\documentclass[10pt,fleqn]{article}
\usepackage[paperwidth=15.5cm,paperheight=23.5cm,
            left=1.5cm,right=1.5cm,top=2cm,bottom=2cm]{geometry}
\usepackage{amsmath,amssymb,amsthm}      
\usepackage{graphicx,graphics,subfigure,float,caption2}
\usepackage{booktabs}
\usepackage{times,cases,color}
\usepackage[dvipdfm,linkcolor=blue,pdfstartview=FitH,
            CJKbookmarks=true,bookmarksnumbered=true,bookmarksopen=true,
            colorlinks=true, 
            pdfborder=000,   
            citecolor=blue]{hyperref}

\newtheorem{theorem}{Theorem}

\newtheorem{remark}{Remark}
\newtheorem{example}{Example}
\newtheorem{proposition}{Proposition}

\numberwithin{equation}{section}
\graphicspath{{figure/}}
\begin{document}
\title{Polynomial Spectral collocation Method for Space Fractional Advection-Diffusion Equation}
\date{}
\author{WenYi Tian  \quad Weihua Deng \quad Yujiang Wu \\
\\
\itshape{School of Mathematics and Statistics, Lanzhou University, } \\
\itshape{Lanzhou 730000, People's Republic of China}}
\maketitle

\fontsize{9.7pt}{11.6pt}\selectfont
\begin{abstract}
This paper discusses the spectral collocation method for numerically
solving nonlocal problems: one dimensional space fractional
advection-diffusion equation; and two dimensional linear/nonlinear
space fractional advection-diffusion equation. The differentiation
matrixes of the left and right Riemann-Liouville and Caputo
fractional derivatives are derived for any collocation points within
any given interval. The stabilities of the one dimensional
semi-discrete and full-discrete schemes are theoretically
established. Several numerical examples
 with different boundary conditions are computed to testify the efficiency of the
numerical schemes and confirm the exponential convergence;  the
physical simulations for L\'evy-Feller advection-diffusion equation
are performed; and the eigenvalue distributions of the iterative
matrix for a variety of systems are displayed to illustrate the
stabilities of the numerical schemes in more general cases.

\end{abstract}
\section{Introduction}\label{sec:1}
During the last decades, it was pointed out that the fractional calculus is more suitable to describe the memory and hereditary properties of materials and processes than the classical one, and hundreds of engineers and scientists are enrolled in the investigation of the fractional calculus and fractional differential equations due to their vast potential of applications in various practical fields, especially, for characterizing anomalous dynamics. In fact, these applications cross diverse disciplines, including physics, chemistry, biology, polymer, mechanical engineering, signal processing, systems identification, control theory, finance, etc \cite{Metzler:00a,Metzler:04,Podlubny:02}.

The space fractional derivative with order $\alpha \in (1,2)$ can
characterize superdiffusion of anomalous diffusion. The anomalous
diffusion has the non-linear growth of the mean squared displacement
in the course of time with the power-law form  $\langle
x^2(t)\rangle\sim \kappa_\alpha t^\alpha$, where  $\kappa_\alpha$ is
the diffusion coefficient; and the superdiffusion corresponds to
$\alpha \in (1,2)$ being exactly the order of space fractional
derivative \cite{Metzler:00a}. The space fractional
advection-diffusion equation can effectively describe superdiffusion
with advection motion. This paper discusses the polynomial spectral
collocation method for the following space fractional
advection-diffusion equation
\begin{equation}\label{eq:0.1}
  \frac{\partial u(\mathbf{x},t)}{\partial t} =\kappa_\alpha\sum_{s=1}^{d}\nabla^{\alpha}_{x_{s}}u^m(\mathbf{x},t)
  -\pmb{\nu_{\alpha}} \cdot \nabla u(\mathbf{x},t)+f(\mathbf{x},t),\quad (\mathbf{x},t)\in\Omega\times(0,T],
\end{equation}
where $\kappa_\alpha>0$ is the generalized diffusivity,
$\pmb{\nu_{\alpha}}$ is the velocity vector and all its components
are positive, ~$d=1$ and $2$, respectively, corresponds to the one
and two dimensional problems we will discuss, $m$ is a nature number
and $m=1$ corresponds to the linear case, and $u(\mathbf{x},t)$ is
the unknown concentration function on the bounded domain
$\Omega=(-1,1)^{d}$, $f(\mathbf{x},t)$ denotes a source, $\nabla$ is
the usual gradient operator,
 and
\begin{equation}\label{eq:0.11}
  \nabla^{\alpha}_{x_{s}}=\big\{p~_{-1}D_{x_{s}}^{\alpha}+q~_{x_s}D_1^{\alpha}\big\},
\end{equation}
with $p,q\ge0,p+q=1$, and $_{-1}D_{x_{s}}^{\alpha}$ and
$_{x_{s}}D_{1}^{\alpha}$ are left and right Riemann-Liouville
fractional partial derivatives of order $\alpha~(1<\alpha<2)$ with
respect to the $s$-th component $x_s$ of $\mathbf{x}$, respectively,
being defined by
\begin{equation}
  \begin{split}\label{eq:0.2}
    & _{-1}D_{x_{s}}^{\alpha}u(\textbf{x},t)
      =\frac{1}{\Gamma(2-\alpha)}\frac{\mathrm{d}^2}{\mathrm{d}{x_{s}}^2}
      \int^{x_{s}}_{-1}\frac{u(\mathbf{x}_\xi,t)}{({x_{s}}-\xi)^{\alpha-1}}\mathrm{d}\xi,\\
    & _{x_{s}}D_1^{\alpha}u(\textbf{x},t)
      =\frac{1}{\Gamma(2-\alpha)}\frac{\mathrm{d}^2}{\mathrm{d}{x_{s}}^2}
      \int_{x_{s}}^1\frac{u(\mathbf{x}_\xi,t)}{(\xi-x_s)^{\alpha-1}}\mathrm{d}\xi,
  \end{split}
\end{equation}
and $\mathbf{x}_\xi$ represents $\mathbf{x}$ with its $s$-th
component replaced by $\xi$.

The space fractional advection-diffusion equation (\ref{eq:0.1})
arises naturally in many practical problems, it can describe the
probability distribution of the particles having advection and
superdiffusion, such as, the nonlocal flows and non-Fickian flows in
the porous media \cite{Pablo:11}. There are some models which are
the special cases of (\ref{eq:0.1}). Chaves presents a one
dimensional fractional advection-diffusion equation \cite{Chaves:98}
(i.e., $d=1$ in (\ref{eq:0.1})) used to investigate the mechanism of
the superdiffusion. A governing equation of stable random walks is
developed in \cite{Benson:00}. Ervin and Roop derive the steady
state fractional advection dispersion equation in $\mathbb{R}^d$ via
a continuous time random walk (CTRW) model \cite{Ervin:07a}. For
more applications of (\ref{eq:0.1}) one can see the recent reviews
\cite{Metzler:00a,Zaslavsky:02,Metzler:04} and references therein.

Several analytical methodologies, such as, Laplace transform, Mellin
transform, Fourier transform, are restored to obtain the analytical
solutions of the fractional equations by many authors
\cite{Podlubny:99,Metzler:00a,Zaslavsky:02,Metzler:04}. Similar to
the classical differential equations, they mainly work for linear
fractional differential equations with constant coefficients, but be
invalid for the equations with variable coefficients and nonlinear
equations; and even more the solutions obtained by analytical
methods are usually transcendental functions or infinite series because
of the nonlocal properties of fractional operators. So, in many
cases the more reasonable option is to find the numerical solutions.
During the recent years, much more efforts have been devoted to the
numerical investigations of fractional integral and differential
equations, such as, the finite difference method
\cite{Deng:11,Liu:07,Meerschaert:03,Meerschaert:06}, finite element method
\cite{Deng:08,Ervin:06,Ervin:07a,Ervin:07,Zhang:10}, and spectral Galerkin
method \cite{X.J.Li:09,X.J.Li:10}, etc.

It is well known that the challenge of solving fractional
differential equations essentially comes from the nonlocal
properties of  fractional derivatives. As a `global' numerical
method, spectral method seems to be a nature choice for obtaining
high order numerical schemes of solving fractional differential
equations. Spectral collocation method has its special advantage for
solving nonlinear problems, the aim of this paper is to design the
polynomial spectral collocation algorithm for solving space
fractional advection-diffusion equation (\ref{eq:0.1}) in one and
two dimensions.

The rest of the paper is organized as follows. In Section
\ref{sec:2} we derive the differentiation matrixes for
Riemann-Liouville and Caputo fractional derivatives of order
$\alpha~(n-1<\alpha<n)$ within some given interval. Section
\ref{sec:3} proposes the semi-discrete and full-discrete schemes for
the space fractional advection-diffusion equation in one dimension
and discusses their numerical stabilities. Section \ref{sec:4}
presents the numerical schemes for the linear and nonlinear
fractional advection-diffusion equations in two dimensions, and the
extensive numerical results are displayed in Section \ref{sec:5} to
confirm the effectiveness of the spectral collocation method for
solving space fractional advection-diffusion equations.
\section{Preliminaries}\label{sec:2}

This section focuses on deriving the differentiation matrixes for
fractional derivatives, and making some remarks and presenting the
propositions of the derived matrixes. This is the first and foremost
job we need to do for applying spectral collocation method to solve
fractional problems. The definitions of fractional derivatives and integral of order $\alpha(n-1<\alpha<n)$ are
as follows \cite{Podlubny:99}:
\begin{itemize}
  \item left Riemann-Liouville fractional derivative:
    \begin{equation*}
      _aD_x^{\alpha}u(x)=\frac{1}{\Gamma(n-\alpha)}\frac{\mathrm{d}^n}{\mathrm{d}
          x^n}\int_a^x\frac{u(\xi)}{(x-\xi)^{\alpha-n+1}}\mathrm{d}\xi
    \end{equation*}
  \item right Riemann-Liouville fractional derivative:
    \begin{equation*}
      _xD_b^{\alpha}u(x)=\frac{(-1)^n}{\Gamma(n-\alpha)}\frac{\mathrm{d}^n}{\mathrm{d}
          x^n}\int_x^b\frac{u(\xi)}{(\xi-x)^{\alpha-n+1}}\mathrm{d}\xi
    \end{equation*}
  \item left Caputo fractional derivative:
    \begin{equation*}
      _a^CD_x^{\alpha}u(x)=\frac{1}{\Gamma(n-\alpha)}\int_a^x
      \frac{u^{(n)}(\xi)}{(x-\xi)^{\alpha-n+1}}\mathrm{d}\xi
    \end{equation*}
  \item right Caputo fractional derivative:
    \begin{equation*}
      _x^CD_b^{\alpha}u(x)=\frac{(-1)^n}{\Gamma(n-\alpha)}
      \int_x^b\frac{u^{(n)}(\xi)}{(\xi-x)^{\alpha-n+1}}\mathrm{d}\xi
    \end{equation*}
  \item left fractional integral:
    \begin{equation*}
      _aD_x^{-\alpha}u(x)=\frac{1}{\Gamma(\alpha)}
      \int_a^x\frac{u(\xi)}{(x-\xi)^{1-\alpha}}\mathrm{d}\xi
    \end{equation*}
  \item right fractional integral:
    \begin{equation*}
      _xD_b^{-\alpha}u(x)=\frac{1}{\Gamma(\alpha)}
      \int_x^b\frac{u(\xi)}{(\xi-x)^{1-\alpha}}\mathrm{d}\xi.
    \end{equation*}
\end{itemize}
From the definitions of Riemann-Liouville derivatives, it is easily
concluded that
\begin{equation}\label{eq:1.10}
  \begin{split}
    & _aD_x^{\alpha}[(x-a)^{\gamma}]=\frac{\Gamma(\gamma+1)}{\Gamma(\gamma+1-\alpha)}
      (x-a)^{\gamma-\alpha},~ \gamma>-1,\\
    & _xD_b^{\alpha}[(x-b)^{\gamma}]=\frac{(-1)^{\gamma}\Gamma(\gamma+1)}{\Gamma
      (\gamma+1-\alpha)}(b-x)^{\gamma-\alpha},~ \gamma>-1.
  \end{split}
\end{equation}
For the left Caputo derivative of $(x-a)^{\gamma}$ and right Caputo
derivative of $(x-b)^{\gamma}$, we have the same formulae as
(\ref{eq:1.10}) except for $\gamma=0,\ldots,n-1$, and both of them
are zeros when $\gamma$ taking these values.


\subsection{Ways to Evaluate Fractional Derivatives of Functions Approximated by Polynomials}
There are two ways to evaluate the fractional derivatives of some
function $u_{_N}$, being the polynomial approximation of $u$: one is
to
 expand $u$ by some special orthogonal polynomials first and then use
the formulae of fractional derivatives of some particular orthogonal
polynomials; another one is to do the Lagrange interpolation for $u$
first and then derive the differentiation matrix based on the
expanded interpolation function.

For the first way, we can apply the following formulae for
Riemann-Liouville fractional derivatives \cite{Askey:75}
\begin{equation}
  \begin{split}
    & (1+x)^{\delta+\alpha}\frac{P_n^{(\gamma-\alpha,\delta+\alpha)}(x)}{P_n^{(\delta+\alpha,\gamma-\alpha)}(1)}
      =\frac{\Gamma(\delta+\alpha+1)}{\Gamma(\delta+1)\Gamma(\alpha)}\int_{-1}^x(1+\xi)^{\delta}
      \frac{P_n^{(\gamma,\delta)}(\xi)}{P_n^{(\delta,\gamma)}(1)}(x-\xi)^{\alpha-1}\mathrm{d}\xi, \\
    & (1-x)^{\gamma+\alpha}\frac{P_n^{(\gamma+\alpha,\delta-\alpha)}(x)}{P_n^{(\gamma+\alpha,\delta-\alpha)}(1)}
      =\frac{\Gamma(\gamma+\alpha+1)}{\Gamma(\gamma+1)\Gamma(\alpha)}\int_x^1(1-\xi)^{\gamma}
      \frac{P_n^{(\gamma,\delta)}(\xi)}{P_n^{(\gamma,\delta)}(1)}(\xi-x)^{\alpha-1}\mathrm{d}\xi,
  \end{split}
\end{equation}
where $\gamma,\delta>-1,~\alpha>0,~-1<x<1$, and
$P_n^{(\gamma,\delta)}(x)$ is the Jacobi polynomial of degree $n$
with respect to the weight function $w(x)=(1-x)^\gamma(1+x)^\delta$.
Applying the properties of Riemann-Liouville fractional derivatives
${_aD_x^\alpha}{_aD_x^{-\alpha}}=I$ and
${_xD_b^\alpha}{_xD_b^{-\alpha}}=I$, the above formulae can be
rewritten as
\begin{equation}
  \begin{split}
    & {_{-1}}D_x^{\alpha}\Big((1+x)^{\delta+\alpha}P_n^{(\gamma-\alpha,\delta+\alpha)}(x)\Big)
      =\frac{\Gamma(n+\delta+\alpha+1)}{\Gamma(n+\delta+1)}(1+x)^{\delta}P_n^{(\gamma,\delta)}(x), \\
    & {_{x}}D_1^{\alpha}\Big((1-x)^{\gamma+\alpha}P_n^{(\gamma+\alpha,\delta-\alpha)}(x)\Big)
      =\frac{\Gamma(n+\gamma+\alpha+1)}{\Gamma(n+\gamma+1)}(1-x)^{\gamma}P_n^{(\gamma,\delta)}(x).
  \end{split}
\end{equation}
This means that the functions
$\{(1+x)^{\delta+\alpha}P_n^{(\gamma-\alpha,\delta+\alpha)}(x)\}$ or
$\{(1-x)^{\gamma+\alpha}P_n^{(\gamma+\alpha,\delta-\alpha)}(x)\}$
can be chosen as bases to expand the function $u$. Actually, these
basis functions are some form of the generalized Jacobi polynomials
given in \cite{Guo:09}, which form a complete orthogonal system in
some weighted $L^2([-1,1])$ space. If the fractional derivative in
(\ref{eq:0.1}) is one sided Riemann-Liouville fractional derivative,
namely, $q=0$ or $p=0$ in (\ref{eq:0.11}), these techniques work
very well; we can choose
$\{(1+x)^{\delta+\alpha}P_n^{(\gamma-\alpha,\delta+\alpha)}(x)\}$ as
basis functions to expand $u$ when $q=0$ and use
$\{(1-x)^{\gamma+\alpha}P_n^{(\gamma+\alpha,\delta-\alpha)}(x)\}$ as
bases when  $p=0$.

For the general form of (\ref{eq:0.1}), including one sided ($p\cdot
q=0$) and two sided cases ($p\cdot q \not=0$), we have to seek the
other way to derive the differentiation matrix. First we do the
Lagrange interpolation of $u(x,t)$ at
Legendre-Gauss-Lobatto/Chebyshev-Gauss-Lobatto points
$\{x_i:~i=0,\ldots,N\}$, and expand it as a power function of
$(x-a)$ ( or $(x-b)$ ); then using (\ref{eq:1.10}) leads to the
evaluation of the fractional derivatives of $u_{_N}$. The
Legendre-Gauss-Lobatto and Chebyshev-Gauss-Lobatto points are the
scaled (from $[-1,1]$ into $[a,b]$) roots of the polynomials
$(1-x^2)P_N^{'}(x)$ and $(1-x^2)T_N^{'}(x)$ respectively, where
$P_N(x)$ and $T_N(x)$ are the Legendre and Chebyshev polynomials of
degree $N$ respectively.
Generally, we need to use numerical methods to get the
Legendre-Gauss-Lobatto points; but for Chebyshev-Gauss-Lobatto
points, there exists the formula $x_i=-\cos(\frac{\pi
i}{N}),~i=0,\ldots,N$ in the interval $[-1,1]$.  The Lagrange
interpolation of $u(x,t)$ is
\begin{equation}\label{eq:1.13}
  u_{_N}(x,t)=\sum_{i=0}^Nu(x_i,t)l_i(x),\qquad l_i(x)=\prod_{j=0,j\neq
  i}^N\frac{x-x_j}{x_i-x_j},
\end{equation}
and the Lagrange polynomials can also be written as
\begin{equation}
  l_i(x)=\frac{Q(x)}{(x-x_i)Q^{'}(x_i)},
\end{equation}
where $Q(x)=(1-x^2)P_N^{'}(x)$ for Legendre-Gauss-Lobatto points or
$Q(x)=(1-x^2)T_N^{'}(x)$ for Chebyshev-Gauss-Lobatto points.
 In the next subsection, we present the details of
deriving the differentiation matrixes and discuss their properties.
\subsection{Differentiation Matrixes for Fractional Derivatives}
Although we choose Legendre-Gauss-Lobatto/Chebyshev-Gauss-Lobatto
points as the collocation points for designing the numerical
schemes. The differentiation matrixes will be derived on any
collocation points $\{x_i:~i=0,\ldots,N\}$ within the given interval
$[a,b]$ for the left and right Riemann-Liouville and Caputo
fractional derivatives of order $\alpha~(n-1<\alpha<n)$, where $N\ge
n$ is assumed. First we use the method given in \cite{Henrici:79} to
expand $l_i(x)$ as
\begin{equation}\label{eq:1.11}
  l_i(x)=\prod_{j=0,j\neq i}^N\frac{x-x_j}{x_i-x_j}=\sum_{k=0}^Nc_{i,k}(x-a)^{N-k}.
\end{equation}
The idea of the method is briefly described as follows. The nodal
polynomial can be expanded as
\begin{equation}
  p(z)=\prod_{i=1}^m(z-z_i)=z^m+c_1z^{m-1}+\cdots+c_m.
\end{equation}
Consider the reciprocal polynomial
\begin{equation}
    q(z)=z^mp(\frac{1}{z})=\prod\limits_{i=1}^m(1-zz_i)=1+c_1z+\cdots+c_mz^m,
\end{equation}
and
\begin{equation}
    q'(z)=\sum\limits_{i=1}^m(-z_i)\prod\limits_{j=1,j\neq i}^m(1-zz_j)=c_1+2c_2z+\cdots+mc_mz^{m-1},
\end{equation}
then
\begin{equation}
    \frac{q^{'}(z)}{q(z)}=\sum_{i=1}^m\frac{-z_i}{1-zz_i}=-\sum_{i=1}^mz_i\sum_{j=0}^{\infty}
    z_i^jz^j=-\sum_{k=1}^{\infty}\big(\sum_{i=1}^mz_i^k\big)z^{k-1}.
\end{equation}
Denoting $s_k=\sum\limits_{i=1}^mz_i^k$, then the coefficients $c_i$ can be determined by comparing the coefficients of the above equality, which obtains
\begin{equation*}
  c_0=1,~c_i=-\frac{1}{i}(s_ic_0+s_{i-1}c_1+\cdots+s_1c_{i-1}), \quad i=1,2,\ldots,m.
\end{equation*}

\begin{remark}
Due to the limits of the machine precision, the errors of the
coefficients in (\ref{eq:1.11}) calculated by computer may explode
with $N$ being large. For $N=15,20,25,30,35$, the max errors of
the coefficients with respect to the Legendre-Gauss-Lobatto points
are $1.30\times10^{-8},1.46\times10^{-5},9.28\times10^{-3},10.33,5304$.
In order to ensure the numerical accuracy, we can use multiple
precision packages, e.g., the Multiple Precision Toolbox for Matlab\cite{MPTM}.
\end{remark}
\begin{remark}
  The coefficients in (\ref{eq:1.11}) can also be obtained from a Vandermonde system which is formulated by substituting the $N+1$ collocation points $x_i$ into (\ref{eq:1.11}).
\end{remark}
Thanks to the equalities (\ref{eq:1.10}), i.e.,
$_aD_x^{\alpha}[(x-a)^{\gamma}]=\frac{\Gamma(\gamma+1)}
{\Gamma(\gamma+1-\alpha)}(x-a)^{\gamma-\alpha}$ and (\ref{eq:1.11}),
the left Riemann-Liouville fractional derivative of $u_{_N}(x,t)$ in
(\ref{eq:1.13}) becomes
\begin{equation}
  \begin{split}
    _aD_x^{\alpha}u_{_N}(x,t)&= {_a}D_x^{\alpha}\sum_{i=0}^Nu(x_i,t)l_i(x)\\
    &=\sum_{i=0}^Nu(x_i,t)\ {_a}D_x^{\alpha}\Big(\sum_{k=0}^Nc_{i,k}(x-a)^{N-k}\Big)\\
    &=\sum_{i=0}^N\Big[\sum_{k=0}^Nc_{i,k}\frac{\Gamma(N-k+1)}{\Gamma(N-k+1-\alpha)}
    (x-a)^{N-k-\alpha}\Big]u(x_i,t).
  \end{split}
\end{equation}
Then the differentiation matrix ${_L}D^{(\alpha)}$ of the left
Riemann-Liouville fractional derivative is
\begin{equation}\label{eq:1.6.0}
  {_L}D_{ji}^{(\alpha)}=\sum_{k=0}^Nc_{i,k}\frac{\Gamma(N-k+1)}{\Gamma(N-k+1-\alpha)}
    (x_j-a)^{N-k-\alpha},\quad i,j=0,\ldots,N.
\end{equation}
For the right Riemann-Liouville derivative of $u_{_N}(x,t)$, it
follows from (\ref{eq:1.10}) that
\begin{equation}
  \begin{split}\label{eq:1.12}
    l_i(x)&=\prod_{j=0,j\neq i}^N\frac{x-x_j}{x_i-x_j}=\sum_{k=0}^Nd_{i,k}(x-b)^{N-k}\\
    _xD_b^{\alpha}u_{_N}(x,t)&= {_xD_b^{\alpha}}\sum_{i=0}^Nu(x_i,t)l_i(x)\\
    &=\sum_{i=0}^Nu(x_i,t)\ _xD_b^{\alpha}\Big(\sum_{k=0}^Nd_{i,k}(x-b)^{N-k}\Big)\\
    &=\sum_{i=0}^N\Big[\sum_{k=0}^Nd_{i,k}\frac{(-1)^{N-k}\Gamma(N-k+1)}{\Gamma(N-k+1-\alpha)}
    (b-x)^{N-k-\alpha}\Big]u(x_i,t).
  \end{split}
\end{equation}
Thus the differentiation matrix ${_R}D^{(\alpha)}$ of the right
Riemann-Liouville fractional derivative is
\begin{equation}\label{eq:1.6.1}
  {_R}D_{ji}^{(\alpha)}=\sum_{k=0}^Nd_{i,k}\frac{(-1)^{N-k}\Gamma(N-k+1)}{\Gamma(N-k+1-\alpha)}
    (b-x_j)^{N-k-\alpha},\quad i,j=0,\ldots,N.
\end{equation}
By the same approach for Riemann-Liouville fractional derivatives and notice that for $n-1<\alpha<n$,
\begin{equation}
    _a^CD_x^{\alpha}[(x-a)^{\gamma}]={_x^C}D_b^{\alpha}[(x-b)^{\gamma}]=0,~
    \gamma=0,\ldots,n-1,
\end{equation}
we obtain the differentiation matrixes of the left and right Caputo fractional derivatives,
\begin{align}\label{eq:1.6.3}
   & {_L^C}D_{ji}^{(\alpha)}=\sum_{k=0}^{N-n}c_{i,k}
     \frac{\Gamma(N-k+1)}{\Gamma(N-k+1-\alpha)}(x_j-a)^{N-k-\alpha},\quad i,j=0,\ldots,N,\\
   & {_R^C}D_{ji}^{(\alpha)}=\sum_{k=0}^{N-n}d_{i,k}\frac{(-1)^{N-k}\Gamma(N-k+1)}
     {\Gamma(N-k+1-\alpha)}(b-x_j)^{N-k-\alpha},\quad i,j=0,\ldots,N.
\end{align}
\begin{remark}
The norm of the differentiation matrix ${_L}D^{(\alpha)}$ is
infinity if some point $x_j=a$ is used, and the norm of
${_R}D^{(\alpha)}$ is infinity if some point $x_j=b$ is chosen, this
is because that the Riemann-Liouville fractional derivatives
${_a}D_x^{\alpha}[(x-a)^{\gamma}]$ and
${_x}D_b^{\alpha}[(b-x)^{\gamma}]$ in (\ref{eq:1.10}) are infinity
at $x=a$ and $x=b$ for $\gamma<\alpha$ respectively. But this
situation does not occur in the differentiation matrixes of the left
and right Caputo fractional derivatives.
\end{remark}
\begin{remark}
   For $\alpha\rightarrow n$, the matrix $ {_L^C}D^{(\alpha)} $ in (\ref{eq:1.6.3}) is the differentiation matrix of the integer derivative of order $n$; particularly the differential matrixes (\ref{eq:1.6.2}) and (\ref{eq:3.2}) of the first order derivative on Legendre-Gauss-Lobatto and Chebyshev-Gauss-Lobatto points respectively, used in the numerical algorithm in the sequel, can also be obtained by (\ref{eq:1.6.3}) indeed. It coincides with the property that ${_a^C}D_x^{\alpha}f(x)=\frac{\mathrm{d}^n}{\mathrm{d}x^n}f(x)$ when $\alpha\rightarrow n$ \cite{Podlubny:99}.
\end{remark}
\begin{proposition}
   If the collocation points $\{x_i:~i=0,\ldots,N\}$ within the interval $[a,b]$ satisfy $x_i-a=b-x_{N-i}$ for $i=0,\ldots,N$,
   then for the coefficients in (\ref{eq:1.11}) and (\ref{eq:1.12}) and the differentiation matrixes of fractional derivatives, the following holds
   \begin{itemize}
     \item[(1)] $c_{i,k}=(-1)^{N+k}d_{N-i,k},\quad i,k=0,\ldots,N$;
     \item[(2)] ${_L}D_{ji}^{(\alpha)}={_R}D_{N-j\ N-i}^{(\alpha)},\quad
                 {_L^C}D_{ji}^{(\alpha)}={_R^C}D_{N-j\ N-i}^{(\alpha)},\quad
                 i,j=0,\ldots,N$;
   \end{itemize}
in fact, the second item (2) says that the differentiation matrix of the right Riemann-Liouville(Caputo) fractional derivative comes by a 180 degrees rotation of the differentiation matrix of the left Riemann-Liouville(Caputo) fractional derivative.
\end{proposition}
\begin{proof}
  From the equality
  \begin{equation}
     l_i(x)=\prod_{j=0,j\neq i}^N\frac{x-x_j}{x_i-x_j}=
     \prod_{j=0,j\neq i}^N\frac{(x-a)-(x_j-a)}{(x_i-a)-(x_j-a)}
     =\sum_{k=0}^Nc_{i,k}(x-a)^{N-k},
  \end{equation}
  it is easily observed that
  \begin{equation}
     \prod_{j=0,j\neq i}^N\frac{(x-a)+(x_j-a)}{(x_i-a)-(x_j-a)}
     =\sum_{k=0}^N(-1)^kc_{i,k}(x-a)^{N-k}.
  \end{equation}
 From (\ref{eq:1.12}) and $x_j-a=b-x_{N-j}$, it yields
  \begin{equation}
    \begin{split}
      l_{N-i}(x) & =\sum_{k=0}^Nd_{N-i,k}(x-b)^{N-k}
                   =\prod_{j=0,j\neq N-i}^N\frac{x-x_j}{x_{N-i}-x_j}\\
      & =\prod_{j=0,j\neq i}^N\frac{(x-b)-(x_{N-j}-b)}{(x_{N-i}-b)-(x_{N-j}-b)}\\
      & =(-1)^N\prod_{j=0,j\neq i}^N\frac{(x-b)+(x_j-a)}{(x_i-a)-(x_j-a)}\\
      & =\sum_{j=0,j\neq i}^N(-1)^{N+k}c_{i,k}(x-b)^{N-k}.
    \end{split}
  \end{equation}
  Comparing the coefficients, we obtain $c_{i,k}=(-1)^{N+k}d_{N-i,k}$.
  And by (\ref{eq:1.6.0}) and (\ref{eq:1.6.1}), we have
  \begin{equation}
    \begin{split}
      {_L}D_{ji}^{(\alpha)} & =\sum_{k=0}^Nc_{i,k}
        \frac{\Gamma(N-k+1)}{\Gamma(N-k+1-\alpha)}(x_j-a)^{N-k-\alpha}\\
      & = \sum_{k=0}^N(-1)^{N-k}d_{N-i,k}
        \frac{\Gamma(N-k+1)}{\Gamma(N-k+1-\alpha)}(b-x_{N-j})^{N-k-\alpha}\\
      & = {_R}D_{N-j\ N-i}^{(\alpha)}.
    \end{split}
  \end{equation}
Using the similar way as above, we can also prove
${_L^C}D_{ji}^{(\alpha)}={_R^C}D_{N-j\ N-i}^{(\alpha)}$.
\end{proof}
In \cite{Hesthaven:07}, the differentiation matrixes based on
Legendre-Gauss-Lobatto and Chebyshev-Gauss-Lobatto points of first
order derivative are, respectively, given by
\begin{equation}\label{eq:1.6.2}
  D_{ji}^l=\left\{
  \begin{split}
    & -\frac{N(N+1)}{4}, && i=j=0,\\
    & \frac{P_N(x_j)}{P_N(x_i)}\frac{1}{x_j-x_i}, && i\neq j,\\
    & 0, && i=j\in[1,\ldots,N-1],\\
    & \frac{N(N+1)}{4}, && i=j=N,
  \end{split}\right.
\end{equation}
and
\begin{equation}\label{eq:3.2}
  D_{ji}^c=\left\{
  \begin{split}
    & -\frac{2N^2+1}{6}, && i=j=0,\\
    & \frac{c_j}{c_i}\frac{(-1)^{i+j}}{x_j-x_i}, && i\neq j,\\
    & -\frac{x_i}{2(1-x_i^2)}, && i=j\in[1,\ldots,N-1],\\
    & \frac{2N^2+1}{6}, && i=j=N,
  \end{split}\right.
\end{equation}
where
\begin{equation*}
  c_{i}=\left\{
  \begin{split}
    & 2, && i=0,N,\\
    & 1, && i\in[1,\ldots,N-1].
  \end{split}\right.
\end{equation*}
In the following, we will also use the denotation $D_{ji}$ which
means both $D_{ji}^l$ and $D_{ji}^c$ work over there.
\section{Spectral Collocation Method for Space Fractional \\Advection-Diffusion Equation in One Dimension}\label{sec:3}
We consider the one dimensional linear case of (\ref{eq:0.1}) with
the Dirichlet/Neumann/mixed boundary conditions, i.e., the following
space fractional differential equation
\begin{equation}\left\{
  \begin{split}\label{eq:1.1}
    & \frac{\partial u(x,t)}{\partial t}=\kappa_{\alpha}\nabla^{\alpha}u(x,t)-\nu_{\alpha}\nabla
      u(x,t)+f(x,t),\quad && -1<x<1,~t>0,\\
    & \alpha_1u(-1,t)-\beta_1\frac{\partial u(-1,t)}{\partial x}=g_1(t),\quad && t>0,\\
    & \alpha_2u(1,t)+\beta_2\frac{\partial u(1,t)}{\partial x}=g_2(t),\quad && t>0,\\
    & u(x,0)=h(x),\quad && -1\le x\le 1,
  \end{split}\right.
\end{equation}
where $1<\alpha<2,~\kappa_{\alpha},\nu_{\alpha}>0$, and
$\alpha_i^2+\beta_i^2>0,~\alpha_i,\beta_i\ge0,\, i=1,2$;\,
$\beta_1=\beta_2=0$ corresponds to the Dirichlet boundary condition,
$\alpha_1=\alpha_2=0$ the Neumann one, and otherwise the mixed one.
We will apply polynomial spectral collocation method based on
Legendre-Gauss-Lobatto/Chebyshev-Gauss-Lobatto collocation points to
approximate (\ref{eq:1.1}).
\subsection{Spectral Collocation Method}
The polynomial spectral collocation method for problem
(\ref{eq:1.1}) is to find $u_{_N}\in\mathbb{P}_N$, being the space
of polynomials of degree equal or less than $N$, such that
\begin{equation}\left\{
  \begin{split}
    & \frac{\mathrm{d}u_{_N}(x_j,t)}{\mathrm{d} t}=\kappa_{\alpha}\nabla^{\alpha}u_{_N}(x_j,t)
      -\nu_{\alpha}\nabla u_{_N}(x_j,t)+f(x_j,t),\, 1\le j\le N-1,\\
    & \alpha_1u_{_N}(-1,t)-\beta_1\frac{\partial u_{_N}(-1,t)}{\partial x}=g_1(t),\\
    & \alpha_2u_{_N}(1,t)+\beta_2\frac{\partial u_{_N}(1,t)}{\partial x}=g_2(t),\\
    & u_{_N}(x_j,0)=h(x_j),\quad 0\le j\le N.
  \end{split}\right.
\end{equation}
That is, for the interior points $x_j,~j=1,\ldots,N-1$, excluding
the points $-1$ and $+1$,
\begin{equation}
  \begin{split}
    & \frac{\mathrm{d}u_{_N}(x_j,t)}{\mathrm{d}
      t}=\sum_{i=0}^N\Big(\kappa_{\alpha}\big(p\ {_LD_{ji}^{(\alpha)}}
      +q\ {_RD_{ji}^{(\alpha)}}\big)-\nu_{\alpha}D_{ji}\Big)u_{_N}(x_i,t)+f(x_j,t),
  \end{split}
\end{equation}
and for boundary and initial conditions, the following hold
\begin{equation}
  \begin{split}
    & \alpha_1u_{_N}(-1,t)-\beta_1\sum_{i=0}^ND_{0i}u_{_N}(x_i,t)=g_1(t),\\
    & \alpha_2u_{_N}(1,t)+\beta_2\sum_{i=0}^ND_{Ni}u_{_N}(x_i,t)=g_2(t),\\
    & u_{_N}(x_j,0)=h(x_j),\quad j=0,\ldots,N.
  \end{split}
\end{equation}
Thus, we get a 2-by-2 system for the computation of the boundary values
\begin{equation}
  \begin{split}
    & (\alpha_1-\beta_1D_{00})u_{_N}(-1,t)-\beta_1D_{0N}u_{_N}(1,t)=g_1(t)
      +\beta_1\sum_{i=1}^{N-1}D_{0i}u_{_N}(x_i,t),\\
    & \beta_2D_{N0}u_{_N}(-1,t)+(\alpha_2+\beta_2D_{NN})u_{_N}(1,t)=g_2(t)
      -\beta_2\sum_{i=1}^{N-1}D_{Ni}u_{_N}(x_i,t).
  \end{split}
\end{equation}
There are several ways to discretize (\ref{eq:1.1}) in the time
direction, here we take the $\theta$ scheme. Then the fully discrete
scheme for (\ref{eq:1.1}) is to find $u_{_N}\in \mathbb{P}_N$ such
that for $1\le j\le N-1$, 
\begin{equation}\left\{
  \begin{split}\label{eq:1.100}
    & D_{\tau}U^{k+1}(x_j)=\kappa_{\alpha}\nabla^{\alpha}U^{k+1}_{\theta}(x_j)
      -\nu_{\alpha}\nabla U^{k+1}_{\theta}(x_j)+f^{k+1}_{\theta}(x_j),\\
    & \alpha_1U^{k}(-1)-\beta_1\frac{\partial U^{k}(-1)}{\partial x}=g_1(t_k),\\
    & \alpha_2U^{k}(1)+\beta_2\frac{\partial U^{k}(1)}{\partial x}=g_2(t_k),\\
    & U^0(x_i)=h(x_i), \quad i=0,\ldots,N.
  \end{split}\right.
\end{equation}
where $U^k(x_j)=u_{_N}(x_j,t_k),~f^k(x_j)=f(x_j,t_k)$, $t_k=k\tau$,
$\tau$ is the time step size and $k$ is an integer,
$0\le\theta\le1$, and the notations $D_{\tau}U^{k+1}$ and
$v_{\theta}^{k+1}$ are used as
\begin{equation*}
  D_{\tau}U^{k+1}(x_j)=\frac{U^{k+1}(x_j)-U^k(x_j)}{\tau}, \qquad v_{\theta}^{k+1}=\theta v^{k+1}+(1-\theta)v^{k}.
\end{equation*}
\subsection{Stability}
Taking Legendre-Gauss-Lobatto points as collocation points, we carry
out the stability analysis of (\ref{eq:1.100}) with the Dirichlet
boundary condition, i.e., $\alpha_1=\alpha_2=1$ and
$\beta_1=\beta_2=0$, and naturally we can take
$g_1(t)=g_2(t)\equiv0$. Generally, the discrete inner product and norm are defined as follows,
\begin{equation}
  (u,v)_{_N}=\sum_{i=0}^Nu(x_i)v(x_i)w_i,\qquad \|u\|_{_N}=\sqrt{(u,u)_{_N}},
\end{equation}
where $\{x_i\}$ and $\{w_i\}$ are the Legendre-Gauss-Lobatto points
and the corresponding quadrature weights given by
\begin{equation}\label{eq:1.0}
    (1-x_i^2)P_N^{'}(x_i)=0,\quad w_i=\frac{2}{N(N+1)}\big(P_N(x_i)\big)^{-2}, \quad 0\le i\le N,
\end{equation}
and $-1=x_0<x_1<\cdots<x_i<\cdots<x_{_N}=1$.

It is obvious that $\|\phi\|_{L^2}=\|\phi\|_{_N}$ if $\phi$ is a
polynomial of degree less than $N$. And for any polynomial
$\phi\in\mathbb{P}_N$, the discrete norm is equivalent to the $L^2$
norm \cite{Guo:98}, namely
\begin{equation}\label{eq:1.2}
    \|\phi\|_{L^2}\le\|\phi\|_{_N}\le\sqrt{2+\frac{1}{N}}\|\phi\|_{L^2}.
\end{equation}
If $N$ is a fixed positive integer and $u_{_N}\in\mathbb{P}_N^0$,
being the space of polynomials of degree equal or less than $N$ and
with zero boundary values, then it can be expanded as a combination
of Legendre orthogonal polynomials $\{P_m(x):~m=0,\ldots,N\}$,
\begin{equation}
    u_{_N}(x)=c_0P_0(x)+c_1P_1(x)+\cdots+c_{_N}P_{N}(x),
\end{equation}
and the left Riemann-Liouville fractional derivative of $u_{_N}$ can be expressed as
\begin{equation}
    {_{-1}}D_x^\alpha u_{_N}(x)=c_0\psi_0(x)+c_1\psi_1(x)+\cdots+c_{_N}\psi_{_N}(x),
\end{equation}
where $\psi_m(x)={_{-1}}D_x^\alpha P_m(x)$.

As $\{P_m(x)\}$ are Legendre orthogonal polynomials, and with (\ref{eq:1.2}), it yields
\begin{equation}\label{eq:2.0}
  \begin{split}
	\|u_{_N}\|_{_N}^2 & \ge\|u_{_N}\|_{L^2}^2=\|\sum_{m=0}^Nc_mP_m(x)\|_{L^2}^2\\
	& =\sum_{m=0}^Nc_m^2\|P_m(x)\|_{L^2}^2\ge\frac{N}{2N+1}\sum_{m=0}^Nc_m^2\|P_m(x)\|_{_N}^2.
  \end{split}
\end{equation}
For each $m$, the summation $\sum_{i=1}^{N-1}\psi_m^2(x_i)\omega_i$ is a fixed positive number determined by $N$, and by using (\ref{eq:1.2}) and $\|P_m(x)\|_{L^2}=(m+1/2)^{-1/2}$ (see \cite{Guo:98}), it yields that $\|P_m(x)\|_{_N}$ is positive and bounded. Then there exists a constant $C$ dependent on $N$ such that
\begin{equation}\label{eq:2.4}
    \sum_{i=1}^{N-1}\psi_m^2(x_i)\omega_i \le C\|P_m(x)\|_{_N}^2,
\end{equation}
where $\{x_i\}_{i=0}^N$ and $\{w_i\}_{i=0}^N$ are given by (\ref{eq:1.0}).

Thus with the inequalities (\ref{eq:2.0}), (\ref{eq:2.4}) and $(\sum_{m=0}^Na_m)^2\le(N+1)\sum_{m=0}^Na_m^2$, we have
\begin{equation}
  \begin{split}
    \sum_{i=1}^{N-1} ({_{-1}}D_x^\alpha u_{_N}(x_i))^2\omega_i
      & \le(N+1)\sum_{m=0}^Nc_m^2\sum_{i=1}^{N-1}\psi_m^2(x_i)\omega_i \\
      & \le C(N)\sum_{m=0}^Nc_m^2\|P_m(x)\|_{_N}^2\le C(N)\|u_{_N}\|_{_N}^2.
  \end{split}
\end{equation}
The similar result can be obtained for the right Riemann-Liouville
fractional derivative of $u_{_N}$, thus it yields
\begin{equation}\label{eq:1.14}
    \sum_{i=1}^{N-1} (\nabla^\alpha u_{_N}(x_i))^2\omega_i \le C(N)\|u_{_N}\|_{_N}^2.
\end{equation}
Based on the usual weak formulation of (\ref{eq:1.1}), the
semi-discrete scheme for (\ref{eq:1.1}) with the homogeneous
Dirichlet boundary conditions reads: Find $u_{_N}\in\mathbb{P}_N^0$
such that for any $v\in\mathbb{P}_N^0$ there exists
\begin{equation}\left\{
  \begin{split}\label{eq:1.4}
    & \big(\frac{\partial u_{_N}}{\partial t},v\big)_{N}=\kappa_\alpha\big(\nabla^\alpha
      u_{_N}, v\big)_{N}-v_\alpha\big(\nabla u_{_N}, v\big)_{N}
      +\big(f, v\big)_{N},\\
    & u_{_N}(x_j,0)=h(x_j), \quad j=0,\ldots,N,
  \end{split}\right.
\end{equation}
where $x_j,\,j=0,\ldots,N,$ are Legendre-Gauss-Lobatto points.
\begin{theorem}[Stability of the semi-discrete scheme]
  For any fixed $N$, the semi-discrete scheme (\ref{eq:1.4}) is stable.
\end{theorem}
\begin{proof}
  Taking $v=u_{_N}$ in (\ref{eq:1.4}), we obtain
  \begin{equation}
    \frac{1}{2}\frac{\mathrm{d}}{\mathrm{d}t}\|u_{_N}\|_{_N}^2=\kappa_\alpha\big(\nabla^\alpha u_{_N}, u_{_N}\big)_{N}-v_\alpha\big(\nabla u_{_N}, u_{_N}\big)_{N}
    +\big(f, u_{_N}\big)_{N}.
  \end{equation}
  Since $u_{_N}$ is a polynomial of degree $N$ and $u_{_N}(\pm1,t)=0$, it has
  \begin{equation}
    \big(\nabla u_{_N}, u_{_N}\big)_{N}=\frac{1}{2}\int_{-1}^1\big(u_{_N}^2(\xi,t)\big)_{\xi}\mathrm{d}\xi
    =\frac{1}{2}\big(u_{_N}^2(1,t)-u_{_N}^2(-1,t)\big)=0.
  \end{equation}
  As $u_{_N}(\pm1,t)=0$, by using the Cauchy-Schwartz inequality and (\ref{eq:1.14}), we have
  \begin{equation}\label{eq:1.15}
      \big(\nabla^\alpha u_{_N}, u_{_N}\big)_{N} \le C(N)\|u_{_N}\|_{_N}^2.
  \end{equation}
  The above discussion implies that the following inequality holds for every fixed $N$,
  \begin{equation}
     \frac{\mathrm{d}}{\mathrm{d}t}\|u_{_N}\|_{_N}^2\le C\|u_{_N}\|_{_N}^2
     +\|f\|_{_N}^2,
  \end{equation}
  then from the Gronwall inequality, we have
  \begin{equation*}
     \|u_{_N}(\cdot,T)\|_{_N}\le \exp(\frac{CT}{2})\Big(\|u_{_N}(\cdot,0)\|_{_N}
     +\int_0^T\|f(\cdot,t)\|_{_N}^2\mathrm{d}t\Big)^{1/2}.
  \end{equation*}
\end{proof}
Next we consider the full-discrete scheme for (\ref{eq:1.1}) with homogeneous Dirichlet boundary conditions, for any $k\ge0$ and $U^{k+1}\in \mathbb{P}_N^0$ such that
\begin{equation}\label{eq:1.3}
  \frac{U^{k+1}(x_j)-U^k(x_j)}{\tau}=\kappa_{\alpha}\nabla^{\alpha}U_{\theta}^{k+1}(x_j)
      -\nu_{\alpha}\nabla
      U_{\theta}^{k+1}(x_j)+f_{\theta}^{k+1}(x_j),
\end{equation}
where $x_j,\,j=0,\ldots,N,$ are Legendre-Gauss-Lobatto points.
The stability of the full-discrete scheme is given
by the following theorem.
\begin{theorem}[Stability of the full-discrete scheme]
  For any fixed $N$ and $\frac{1}{2}\le\theta\le1$, the full-discrete scheme (\ref{eq:1.3}) is stable, and
  \begin{equation}\label{eq:1.9}
    \|U^{M}\|_{_N}\le\exp(\frac{CT}{2})\Big(\|U^{0}\|_{_N}^2+
    \tau\sum_{k=0}^M\|f^k\|_{_N}^2\Big)^{1/2},
  \end{equation}
  where $T=M\tau$.
\end{theorem}
\begin{proof}
  We can rewrite (\ref{eq:1.3}) in the following form,
  \begin{equation}
    \begin{split}\label{eq:1.7}
      & \big(\frac{U^{k+1}-U^k}{\tau},v\big)_{N}
        =\kappa_\alpha\big(\nabla^{\alpha}U_{\theta}^{k+1}, v\big)_{N}
        -\nu_\alpha\big(\nabla U_{\theta}^{k+1}, v\big)_{N}+\big(f_{\theta}^{k+1},
        v\big)_{N}.
    \end{split}
  \end{equation}
  Setting $v=U_{\theta}^{k+1}$ in (\ref{eq:1.7}) and using the Cauchy-Schwartz inequality, for $\frac{1}{2}\le\theta\le1$, we obtain
  \begin{equation*}
    \begin{split}
      \big(\frac{U^{k+1}-U^k}{\tau},U_{\theta}^{k+1}\big)_{N}
      &= \frac{1}{\tau}\Big(\theta\|U^{k+1}\|_{_N}^2+(1-2\theta)\big(U^{k+1},U^k\big)_N
        -(1-\theta)\|U^k\|_{_N}^2\Big)\\
      & \ge \frac{1}{2\tau}\Big(\|U^{k+1}\|_{_N}^2-\|U^k\|_{_N}^2\Big).
    \end{split}
  \end{equation*}
  Using (\ref{eq:1.15}), for any fixed $N$, it follows
  \begin{equation*}
    \begin{split}
      \big(\nabla^{\alpha}U_{\theta}^{k+1}, U_{\theta}^{k+1}\big)_{N}
      & \le C\|U_{\theta}^{k+1}\|_{_N}^2
        \le C\big(\theta\|U^{k+1}\|_{_N}^2
        +(1-\theta)\|U^k\|_{_N}^2\big).
    \end{split}
  \end{equation*}
  Since $U_{\theta}^{k+1}$ is a polynomial of degree $N$ and $U_{\theta}^{k+1}(\pm1)=0$, then
  \begin{equation*}
    \begin{split}
      \big(\nabla U_{\theta}^{k+1}, U_{\theta}^{k+1}\big)_{N}
      =\frac{1}{2}\int_{-1}^1\nabla\big(U_{\theta}^{k+1}(\xi)\big)^2\mathrm{d}\xi
      =\frac{1}{2}\big(U_{\theta}^{k+1}(1)\big)^2
      -\frac{1}{2}\big(U_{\theta}^{k+1}(-1)\big)^2=0.
    \end{split}
  \end{equation*}
  For the last term,
  \begin{equation*}
    \begin{split}
      & \big(f_{\theta}^{k+1}, U_{\theta}^{k+1}\big)_{N}
        \le \frac{\theta}{2}\|U^{k+1}\|_{_N}^2+\frac{1-\theta}{2}\|U^k\|_{_N}^2
        +\frac{\theta}{2}\|f^{k+1}\|_{_N}^2+\frac{1-\theta}{2}\|f^k\|_{_N}^2.
    \end{split}
  \end{equation*}
  Combining the above equations, we get
  \begin{equation}\label{eq:1.8}
    \begin{split}
      \frac{1}{2\tau}\Big(\|U^{k+1}\|_{_N}^2-\|U^k\|_{_N}^2\Big)
      \le & C\big(\frac{\theta}{2}\|U^{k+1}\|_{_N}^2+\frac{1-\theta}{2}\|U^k\|_{_N}^2\big) \\
      & +\frac{\theta}{2}\|f^{k+1}\|_{_N}^2+\frac{1-\theta}{2}\|f^k\|_{_N}^2.
    \end{split}
  \end{equation}
  Lastly, summing (\ref{eq:1.8}) from $k=0$ to $k=M-1$, it obtains
  \begin{equation}
      \|U^{M}\|_{_N}^2 \le \|U^0\|_{_N}^2+ C\tau\sum_{k=0}^M\|U^{k}\|_{_N}^2+\tau\sum_{k=0}^M\|f^k\|_{_N}^2.
  \end{equation}
  By the discrete Gronwall inequality, we obtain (\ref{eq:1.9}).
\end{proof}
\section{Space Fractional Advection-Diffusion Equation in Two Dimensions}\label{sec:4}
\subsection{Linear Case of (\ref{eq:0.1})}
We establish the numerical scheme for the linear space fractional
advection-diffusion equation in two dimensions,
\begin{equation}\left\{
  \begin{split}\label{eq:2.1}
    & \frac{\partial u}{\partial t}
      =\kappa_{\alpha}(\nabla^{\alpha}_x+\nabla^{\alpha}_y)u
      -(\nu_{\alpha}^1,\nu_{\alpha}^2)^{\mathrm{T}} \cdot \nabla u+f(x,y,t), && (x,y)\in\Omega,~t>0,\\
    & u(x,y,t)=g(x,y,t), && (x,y)\in\partial\Omega,~t>0,\\
    & u(x,y,0)=h(x,y), && (x,y)\in\bar{\Omega},
  \end{split}\right.
\end{equation}
where
$\Omega=(-1,1)^2,~\kappa_{\alpha}>0,\,\nu_{\alpha}^1=\nu_{\alpha}^2=\nu_{\alpha}>0$.
The Lagrange interpolation of $u$ at the mesh grid points
$\{(x_i,y_j):0\le i,j\le N\}$ is
\begin{equation}
  u_{_N}(x,y,t)=\sum_{i,j=0}^Nu(x_i,y_j,t)l_i(x)l_j(y),
\end{equation}
where $l_i(x)$ and $l_j(y)$ are the Lagrange interpolation
polynomials, and $x_i$ and $y_j$ are
Legendre-Gauss-Lobatto/Chebshev-Gauss-Lobatto points.
\begin{figure}[H]
  \centering
  \includegraphics[scale=0.5]{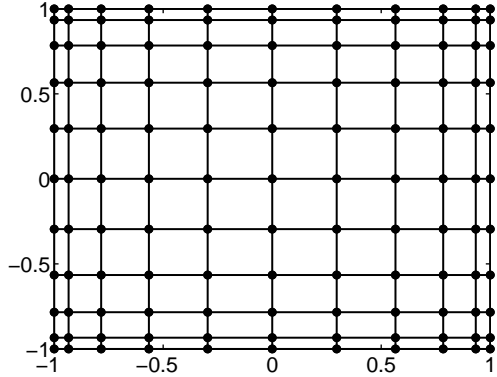}\vspace{-15pt}
  \caption{The Legendre-Gauss-Lobatto collocation grid for $N=10$.}\label{fig:g}
\end{figure}
The spectral collocation method to problem (\ref{eq:2.1}) is to find $u_{_N}\in
\mathbb{P}_N(\Omega)$ such that
\begin{equation}\left\{
  \begin{split}\label{eq:2.2}
    & \frac{\mathrm{d}u_{_N}(x_r,y_s,t)}{\mathrm{d} t}
      =\kappa_{\alpha}(\nabla^{\alpha}_x+\nabla^{\alpha}_y)u_{_N}(x_r,y_s,t)
      -(\nu_{\alpha},\nu_{\alpha})^{\mathrm{T}} \cdot\nabla u_{_N}(x_r,y_s,t)\\
    & \qquad\qquad\qquad\quad\, +f(x_r,y_s,t), \qquad 1\le r,s\le N-1,\\
    & u_{_N}(x_r,y_s,t)=g(x_r,y_s,t),\,\, r=0,N\,( 0\le s\le N) ~\text{and}~ s=0,N\,(0\le r\le N),\\
    & u_{_N}(x_r,y_s,0)=h(x_r,y_s), \,\, ~0\le r,s\le N.
  \end{split}\right.
\end{equation}
Furthermore, for points $\{(x_r,y_s):1\le r,s\le N-1\}$,
\begin{equation}\label{eq:2.3}
  \begin{split}
    \frac{\mathrm{d}u_{_N}(x_r,y_s,t)}{\mathrm{d} t}
    =& \sum_{i=0}^N\big(\kappa_{\alpha}(p\ {_L}D^{\alpha}_{ri}+q\ {_R}D^{\alpha}_{ri})
      -\nu_{\alpha}D_{ri}\big)u_{_N}(x_i,y_s,t)\\
    & +\sum_{j=0}^N\big(\kappa_{\alpha}(p\ {_L}D^{\alpha}_{sj}
      +q\ {_R}D^{\alpha}_{sj})-\nu_{\alpha}D_{sj}\big)u_{_N}(x_r,y_j,t)\\
    & +f(x_r,y_s,t).
  \end{split}
\end{equation}
The values of $u_{_N}$ on boundary points in (\ref{eq:2.3}) are
known, then it leads to an ordinary differential equations of the
$(N-1)^2$ unknown variables. If the unknown variables are arranged
in the sequence as
$[u_{1,1},\ldots,u_{N-1,1},\ldots,u_{1,s},\ldots,u_{N-1,s},
\ldots,u_{1,N-1},\ldots,u_{N-1,N-1}]^{\mathrm{T}},$ where
$u_{r,s}=u_{_N}(x_r,y_s,t),$ then the coefficient matrix
$\mathbf{M}$ of the right hand side is
\begin{equation*}
  \mathbf{M}=\kappa_{\alpha}\big(p(I\otimes D{_L}+D{_L}\otimes I)+q(I\otimes D{_R}+D{_R}\otimes I)\big)-\nu_{\alpha}(I\otimes \tilde{D}+\tilde{D}\otimes I),
\end{equation*}
where
$(D{_L})_{ij}={_L}D^{(\alpha)}_{ij},~(D{_R})_{ij}={_R}D^{(\alpha)}_{ij},~(\tilde{D})_{ij}=D_{ij}$,
$i,j=1,\ldots,N-1$, and ${_L}D^{(\alpha)},~{_R}D^{(\alpha)}$ and $D$
are given by (\ref{eq:1.6.0}), (\ref{eq:1.6.1}) and
(\ref{eq:1.6.2}), (\ref{eq:3.2}), respectively. $I$ is a unit matrix
of order $N-1$, and $\otimes$ stands for Kronecker product.

Using $\theta$ scheme in the time direction for (\ref{eq:2.2}), then
the full-discrete scheme of (\ref{eq:2.1}) is that for $1\le r,s\le
N-1$ the following holds
\begin{equation*}\left\{
  \begin{split}
    & D_{\tau}U^{k+1}(x_r,y_s)=\kappa_{\alpha}(\nabla^{\alpha}_x+\nabla^{\alpha}_y)
      U^{k+1}_{\theta}(x_r,y_s)
      -\nu_{\alpha}\nabla U^{k+1}_{\theta}(x_r,y_s)+f^{k+1}_{\theta}(x_r,y_s),\\
    & U^0(x_r,y_s)=h(x_r,y_s),
  \end{split}\right.
\end{equation*}
and
$U^k(\pm1,y_i)=g(\pm1,y_i,t_k),~U^k(x_i,\pm1)=g(x_i,\pm1,t_k),~i=0,\ldots,N$,
where $U^k(x_r,y_s)=u_{_N}(x_r,y_s,t_k)$,
$f^k(x_r,y_s)=f(x_r,y_s,t_k)$, $\tau$ is the time step size and
$t_k=k\tau$, and the notations $D_{\tau}U^{k+1}$ and
$v_{\theta}^{k+1}$ are used as
\begin{equation*}
  D_{\tau}U^{k+1}(x_r,y_s)=\frac{U^{k+1}(x_r,y_s)-U^k(x_r,y_s)}{\tau}, \qquad v_{\theta}^{k+1}=\theta v^{k+1}+(1-\theta)v^{k}.
\end{equation*}
\subsection{Nonlinear Case of (\ref{eq:0.1})}\label{sec:6}

Our concentration in this subsection is to apply the polynomial
spectral collocation method to seek the numerical solution of the
nonlinear advection-diffusion equation
\begin{equation}\left\{
  \begin{split}\label{eq:4.1}
    & \frac{\partial u}{\partial t}
      =\kappa_{\alpha}(\nabla^{\alpha}_x+\nabla^{\alpha}_y)u^m
      -(\nu_{\alpha},\nu_{\alpha})^{\mathrm{T}}\cdot\nabla u+f(x,y,t), && (x,y)\in\Omega,~t>0,\\
    & u(x,y,t)=g(x,y,t), && (x,y)\in\partial\Omega,~t>0,\\
    & u(x,y,0)=h(x,y), && (x,y)\in\bar{\Omega},
  \end{split}\right.
\end{equation}
where $\Omega=(-1,1)^2$ and the notations are the same as those in
(\ref{eq:2.1}). The spectral collocation method of (\ref{eq:4.1}) is
to find $u_{_N}\in\mathbb{P}_N(\Omega)$ such that for points
$\{(x_r,y_s):1\le r,s\le N-1\}$, there exists
\begin{equation}
  \begin{split}\label{eq:4.2}
    \frac{\mathrm{d}u_{_N}(x_r,y_s,t)}{\mathrm{d} t}
    =& \sum_{i=0}^N\Big(\kappa_{\alpha}(p\ {_L}D^{\alpha}_{ri}+q\ {_R}D^{\alpha}_{ri})
      u_{_N}^m(x_i,y_s,t)-\nu_{\alpha}D_{ri}\ u_{_N}(x_i,y_s,t)\Big)\\
    & +\sum_{j=0}^N\Big(\kappa_{\alpha}(p\ {_L}D^{\alpha}_{sj}
      +q\ {_R}D^{\alpha}_{sj})u_{_N}^m(x_r,y_j,t)-\nu_{\alpha}D_{sj}\ u_{_N}(x_r,y_j,t)\Big)\\
    & +f(x_r,y_s,t),
  \end{split}
\end{equation}
and the boundary and initial conditions
\begin{equation}
  \begin{split}
    & u_{_N}(x_r,y_s,t)=g(x_r,y_s,t),\quad r=0,N\, (0\le s\le N) ~\text{or}~ s=0,N\, (0\le r\le N),\\
    & u_{_N}(x_r,y_s,0)=h(x_r,y_s), \qquad 0\le r,s\le N,
  \end{split}
\end{equation}
where $x_r,~y_s,~r,s=0,\ldots,N$ are the
Legendre-Gauss-Lobatto/Chebyshev-Gauss-Lobatto points. Using the
Crank-Nicholson scheme in (\ref{eq:4.2}) leads to a nonlinear system
\begin{equation}
  \begin{split}\label{eq:4.3}
    F(U^{k+1})=&(E-\frac{\tau}{2}D_I)U^{k+1}-\frac{\tau}{2}D_F(U^{k+1})^m
    -(E+\frac{\tau}{2}D_I)U^{k}-\frac{\tau}{2}D_F(U^{k})^m\\
    &-\frac{\tau}{2}\big(f^{k+1}+f^k\big)=0,
  \end{split}
\end{equation}
where
$$U^{k}=[u_{1,1}^k,\ldots,u_{N-1,1}^k,\ldots,u_{1,s}^k,\ldots,u_{N-1,s}^k,
\ldots,u_{1,N-1}^k,\ldots,u_{N-1,N-1}^k]^{\mathrm{T}},$$
$$u_{r,s}^k=u_{_N}(x_r,y_s,t_k),$$
and
$$f^{k}=[f_{1,1}^k,\ldots,f_{N-1,1}^k,\ldots,f_{1,s}^k,\ldots,f_{N-1,s}^k,
\ldots,f_{1,N-1}^k,\ldots,f_{N-1,N-1}^k]^{\mathrm{T}},$$
\begin{equation*}
  \begin{split}
      f_{r,s}^k= & f(x_r,y_s,t_k)+\big[\kappa_{\alpha}(p\ {_L}D^{(\alpha)}_{r0}+q\ {_R}D^{(\alpha)}_{r0})u_{_N}^m(x_0,y_s,t_k)-\nu_{\alpha}D_{r0}\ u_{_N}(x_0,y_s,t)\big]\\
      & +\big[\kappa_{\alpha}(p\ {_L}D^{(\alpha)}_{rN}+q\ {_R}D^{(\alpha)}_{rN})u_{_N}^m(x_N,y_s,t_k)-\nu_{\alpha}D_{rN}\ u_{_N}(x_N,y_s,t)\big]\\
      & +\big[\kappa_{\alpha}(p\ {_L}D^{(\alpha)}_{s0}+q\ {_R}D^{(\alpha)}_{s0})u_{_N}^m(x_r,y_0,t_k)-\nu_{\alpha}D_{s0}\ u_{_N}(x_r,y_0,t)\big]\\
      & +\big[\kappa_{\alpha}(p\ {_L}D^{(\alpha)}_{sN}+q\ {_R}D^{(\alpha)}_{sN})u_{_N}^m(x_r,y_N,t_k)-\nu_{\alpha}D_{sN}\ u_{_N}(x_r,y_N,t)\big],
  \end{split}
\end{equation*}
and $D_F=\kappa_{\alpha}\big(p(I\otimes D{_L}+D{_L}\otimes
I)+q(I\otimes D{_R}+D{_R}\otimes I)\big),~D_I=-\nu_{\alpha}(I\otimes
\tilde{D}+\tilde{D}\otimes I)$, and
$(D{_L})_{ij}={_L}D^{(\alpha)}_{ij},~(D{_R})_{ij}={_R}D^{(\alpha)}_{ij},~(\tilde{D})_{ij}=D_{ij}$,
$i,j=1,\ldots,N-1$, and $E=I\otimes I$, $I$ is a unit matrix of
order $N-1$, and $\otimes$ stands for Kronecker product.
${_L}D^{(\alpha)},~{_R}D^{(\alpha)}$ and $D$ are given in
(\ref{eq:1.6.0}), (\ref{eq:1.6.1}) and (\ref{eq:1.6.2}),
(\ref{eq:3.2}), respectively.
\section{Numerical Examples}\label{sec:5}
To test the efficiency of the spectral collocation method for the
fractional advection-diffusion equation (\ref{eq:0.1}), several
numerical examples are presented. As it is difficult to find an
analytic solution of fractional differential equations
(\ref{eq:0.1}), in the following examples we assume that some
analytic function $u$ satisfying the given boundary conditions is
the solution of (\ref{eq:0.1}), and the values of the (to be
determined) source function $f$ at collocation points are calculated
by numerical method in our implementation. The key task to compute
the source $f$ is to calculate the Riemann-Liouville fractional
derivatives of order $\alpha~(n-1<\alpha<n)$ of a given function
$v(x)$, and this question can be dealt with by the following
formulae
\begin{equation}
  \begin{split}\label{eq:3.1}
    & _aD_x^{\alpha}v(x)=\sum_{k=0}^{n-1}\frac{v^{(k)}(a)(x-a)^{k-\alpha}}
      {\Gamma(k+1-\alpha)}+\frac{1}{\Gamma(n-\alpha)}
      \int_a^x\frac{v^{(n)}(\xi)}{(x-\xi)^{\alpha-n+1}}\mathrm{d}\xi,\\
    & _xD_b^{\alpha}v(x)=\sum_{k=0}^{n-1}\frac{(-1)^k v^{(k)}(b)(b-x)^{k-\alpha}}
      {\Gamma(k+1-\alpha)}+\frac{(-1)^n}{\Gamma(n-\alpha)}
      \int_x^b\frac{v^{(n)}(\xi)}{(\xi-x)^{\alpha-n+1}}\mathrm{d}\xi,\\
  \end{split}
\end{equation}
the second term of the right hand side of (\ref{eq:3.1}) valued at
collocation points can be determined by numerical integration, and
the Gauss-Lobatto-Jacobi quadrature rule is used in our numerical
experiments. We choose the Crank-Nicolson scheme for the
discretization in the time direction, i.e., $\theta=0.5$, in all the
numerical examples given in this section.
\subsection{Examples for One Dimension}
Four numerical examples in one dimension are given in this section.
In order to verify the accuracy of the space spectral approximation
with less cost, the analytical solution of our problem we select is
a second order polynomial with respect to the time variable $t$, and
the Example \ref{exm:1} is in such case. The numerical results are
measured in the $L^{\infty}$ and $L^2$ norms defined by
\begin{equation*}
  \begin{split}
    & L^{\infty}=\max_{0\le k\le N}\big|u(x_k,\cdot)-u_{_N}(x_k,\cdot)\big|,\\
    & L^2=\Big(\sum_{k=0}^N\big|u(x_k,\cdot)-u_{_N}(x_k,\cdot)\big|^2\omega_k\Big)^{1/2}.
  \end{split}
\end{equation*}
\begin{example}\label{exm:1}
Taking $u(x,t)=(t^2+1)(\cos\pi x+1)$ as the solution of
(\ref{eq:1.1}) with homogeneous boundary conditions since
$u(\pm1,t)=0$. The coefficients in (\ref{eq:1.1}) are chosen as
$\kappa_{\alpha}=\nu_{\alpha}=1,~p=q=0.5$.
\end{example}
\begin{figure}[h]
    \subfigure[$\alpha=1.1$]{
      \begin{minipage}[t]{0.5\linewidth}
        \centering
        \includegraphics[scale=0.4]{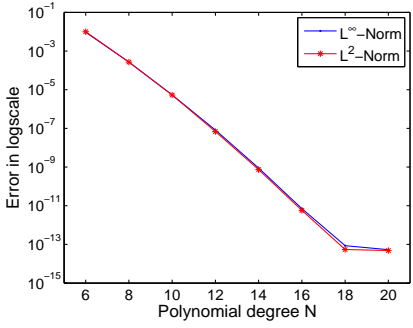}
      \end{minipage}}
    \subfigure[$\alpha=1.3$]{
      \begin{minipage}[t]{0.5\linewidth}
        \centering
        \includegraphics[scale=0.4]{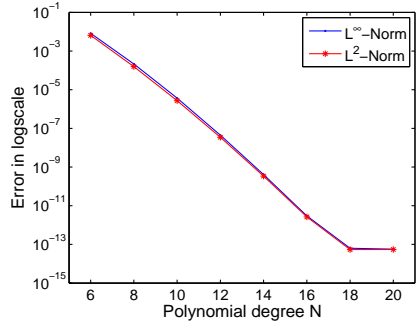}
      \end{minipage}}
    \subfigure[$\alpha=1.7$]{
      \begin{minipage}[t]{0.5\linewidth}
        \centering
        \includegraphics[scale=0.4]{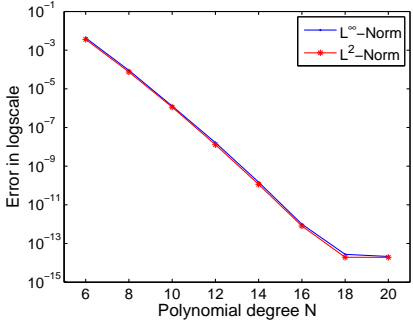}
      \end{minipage}}
    \subfigure[$\alpha=1.9$]{
      \begin{minipage}[t]{0.5\linewidth}
        \centering
        \includegraphics[scale=0.4]{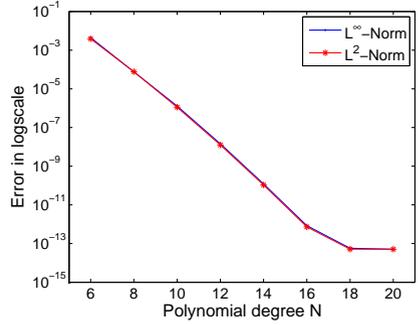}
      \end{minipage}}
    \caption{$L^{\infty}$ and $L^2$ errors to Example \ref{exm:1} approximated on Legendre-Gauss-Lobatto points at $t=1$ for $\alpha=1.1,~1.3,~1.7,~1.9$ with $\tau=0.1$.}\label{fig:a}
\end{figure}
Figure \ref{fig:a} shows the $L^{\infty}$ and $L^2$ errors to Example \ref{exm:1} approximated on Legendre-Gauss-Lobatto points at $t=1$ for $\alpha=1.1,~1.3,~1.7,~1.9$ with $\tau=0.1$, from which we observe that the Legendre spectral collocation method achieves high accuracy for our problem, and the solution converges exponentially.

The $L^{\infty}$ and $L^2$ errors to Example \ref{exm:1}
approximated on Legendre-Gauss-Lobatto and Chebyshev-Gauss-Lobatto
points at $t=1$ for $\alpha=1.5$ with $\tau=0.1$ are presented in
Table \ref{tab:a}, from which we see that the two kinds of
collocation methods have almost the same efficiency. Figure
\ref{fig:b} shows the eigenvalue distribution of the iterative
matrix of the full-discrete scheme to Example \ref{exm:1} for $\alpha=1.5$ with $\tau=0.1$ and $N=6,~12,~18,~24$, the information the figure tells us is that the method we use for (\ref{eq:1.1}) is stable in the numerical trial.
\begin{table}[h]\fontsize{9pt}{12pt}\selectfont
  \begin{center}
    \caption{The $L^{\infty}$ and $L^2$ errors to Example \ref{exm:1} approximated on Legendre-Gauss-Lobatto and Chebyshev-Gauss-Lobatto points at $t=1$ for $\alpha=1.5$ with $\tau=0.1$.}\vspace{5pt}
    \begin{tabular*}{\linewidth}{@{\extracolsep{\fill}}*{2}{r}*{2}{l}*{3}{l}} \midrule
         &  & \multicolumn{2}{l}{Legendre collocation method} & \multicolumn{2}{l}{Chebyshev collocation method} & \\\midrule
         & $N$ & $L^{\infty}$ & $L^2$     & $L^{\infty}$ & $L^2$ & \\\toprule
         & 6 & 4.69749E-03 & 4.00738E-03  & 7.10790E-03 & 6.74624E-03 & \\
         & 8 & 9.66278E-05 & 8.76620E-05  & 1.60074E-04 & 1.49439E-04 & \\
         & 10 & 1.67131E-06 & 1.41412E-06  & 2.17903E-06 & 2.34726E-06 & \\
         & 12 & 1.97558E-08 & 1.71243E-08  & 2.66105E-08 & 2.75979E-08 & \\
         & 14 & 1.86275E-10 & 1.60217E-10  & 2.55300E-10 & 2.51841E-10 & \\
         & 16 & 1.46372E-12 & 1.19140E-12  & 1.79889E-12 & 1.83487E-12 & \\
         & 18 & 1.46549E-14 & 9.88058E-15  & 1.11022E-14 & 1.15594E-14 & \\
         & 20 & 4.66294E-15 & 3.74558E-15  & 5.77316E-15 & 5.14799E-15 & \\
    \midrule
    \end{tabular*}\label{tab:a}
  \end{center}
\end{table}
\begin{figure}[h]
    \subfigure[Legendre collocation method]{
      \begin{minipage}[t]{0.5\linewidth}
        \centering
        \includegraphics[scale=0.4]{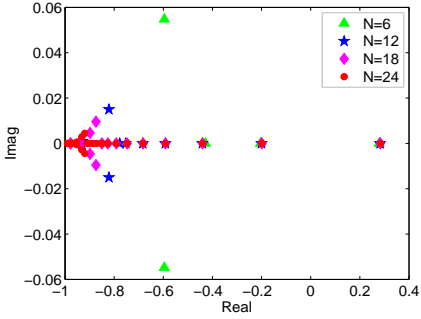}
      \end{minipage}}
    \subfigure[Chebyshev collocation method]{
      \begin{minipage}[t]{0.5\linewidth}
        \centering
        \includegraphics[scale=0.4]{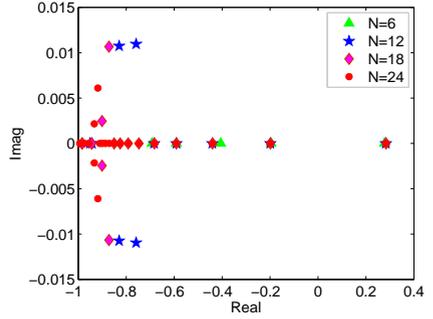}
      \end{minipage}}
    \caption{The eigenvalue distribution of the iterative matrix of the full-discrete scheme of Example \ref{exm:1} for $\alpha=1.5$ with $\tau=0.1$.}\label{fig:b}
\end{figure}

Next we give an example with the mixed boundary conditions.
\begin{example}\label{exm:2}
  Taking $u(x,t)=\exp(\frac{t}{2})\sin x$, $\kappa_{\alpha}=\nu_{\alpha}=1,~p=q=0.5$, and it is a solution of (\ref{eq:1.1}) with the following boundary conditions.
  \begin{align*}
      & u(-1,t)-\frac{\partial u(-1,t)}{\partial x}
        =\exp(\frac{t}{2})(\cos(1)-\sin(1)),\quad t>0,&\\
      & u(1,t)+\frac{\partial u(1,t)}{\partial x}
        =\exp(\frac{t}{2})(\cos(1)+\sin(1)), ~~\qquad t>0.&
  \end{align*}
\end{example}
In Figure \ref{fig:c}, we plot the $L^{\infty}$ and $L^2$ errors to
Example \ref{exm:2} approximated on Legendre-Gauss-Lobatto and
Chebyshev-Gauss-Lobatto points at $t=1$ for
$\alpha=1.1,~1.3,~1.5,~1.7,~1.9$ with $\tau=1/10000$.
\begin{figure}[h!]\centering
    \subfigure[$\alpha=1.1$]{
      \begin{minipage}[t]{0.3\linewidth}
        \centering
        \includegraphics[scale=0.28]{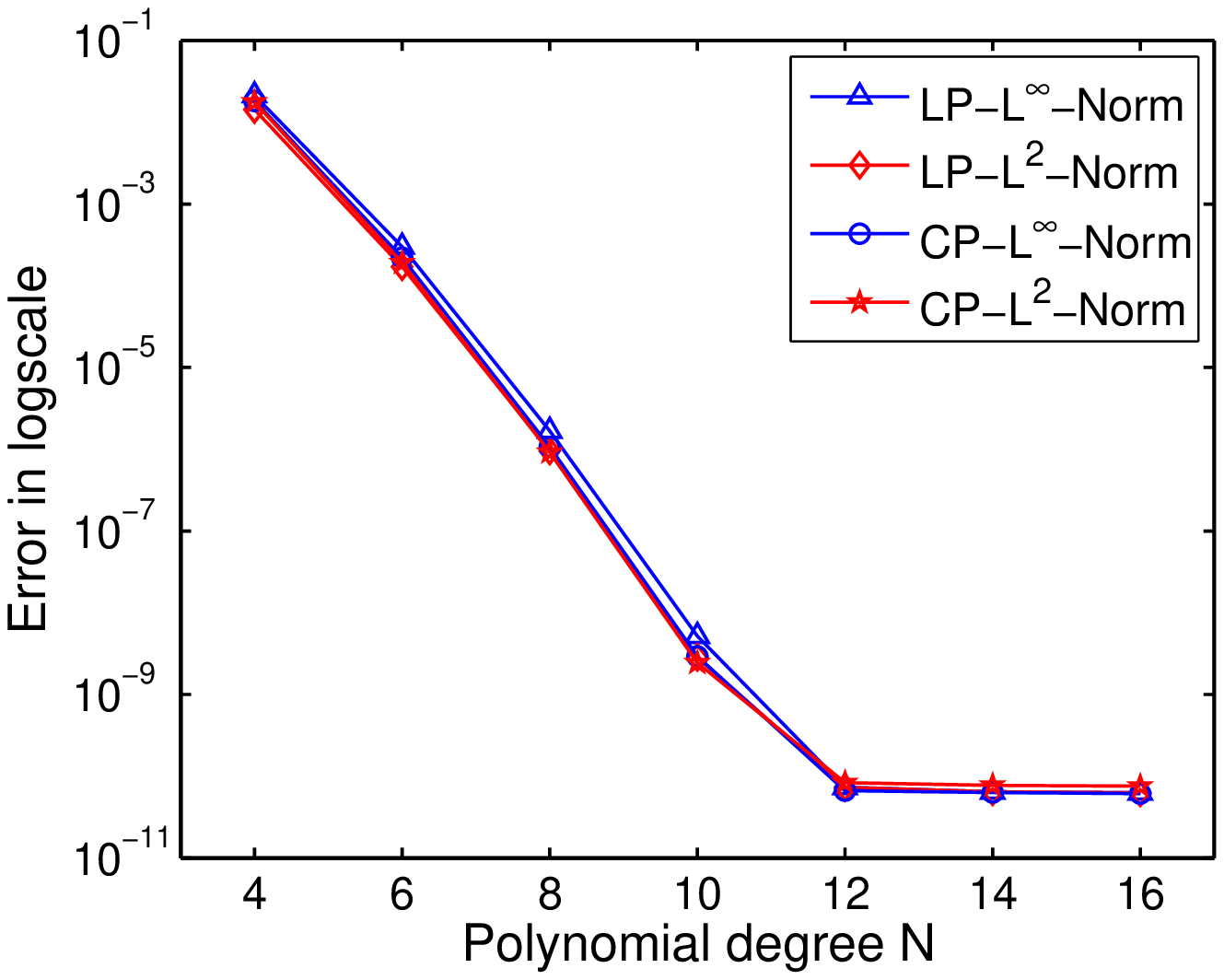}
      \end{minipage}}\hspace{10pt}
    \subfigure[$\alpha=1.3$]{
      \begin{minipage}[t]{0.3\linewidth}
        \centering
        \includegraphics[scale=0.28]{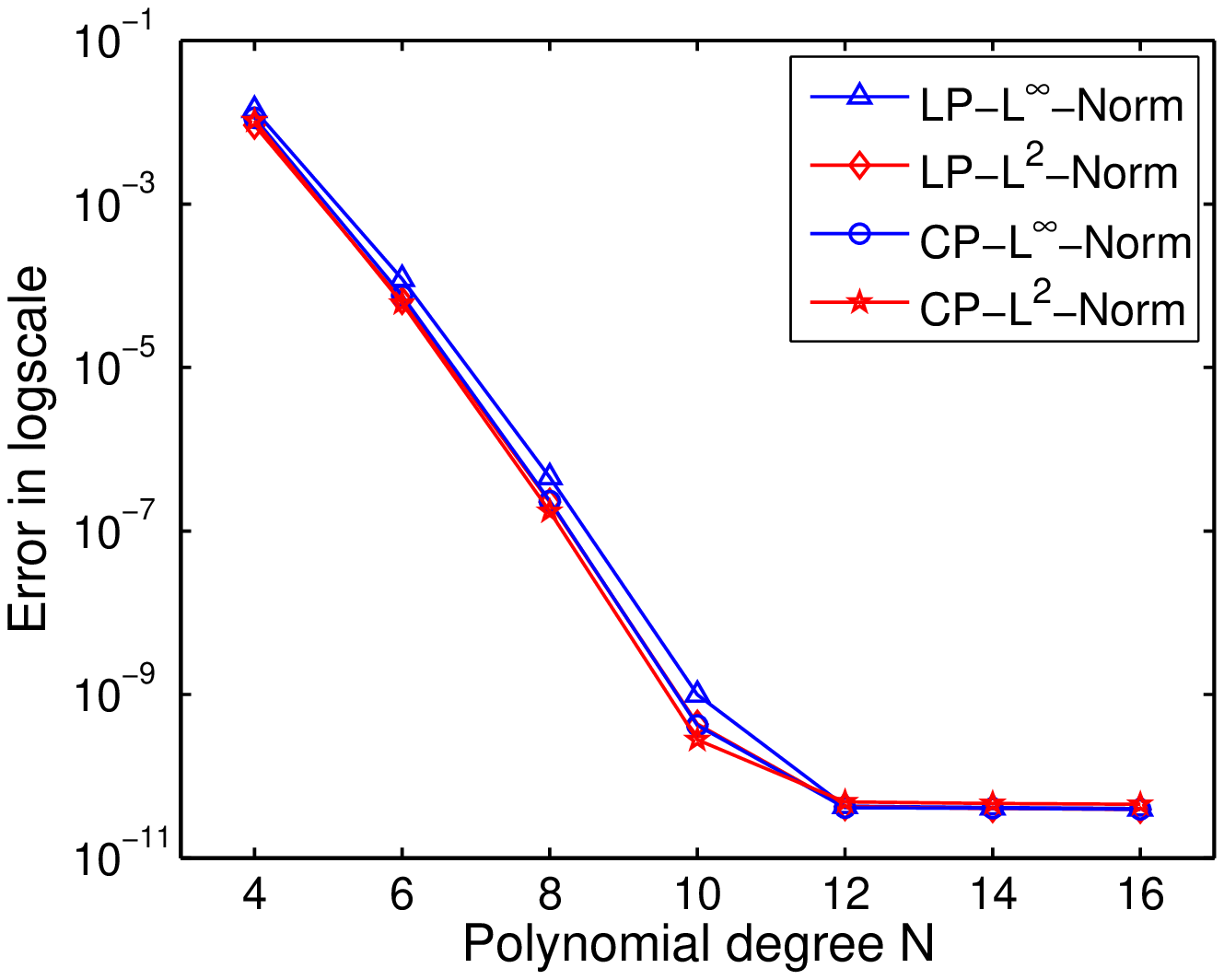}
      \end{minipage}}\hspace{10pt}
    \subfigure[$\alpha=1.5$]{
      \begin{minipage}[t]{0.3\linewidth}
        \centering
        \includegraphics[scale=0.28]{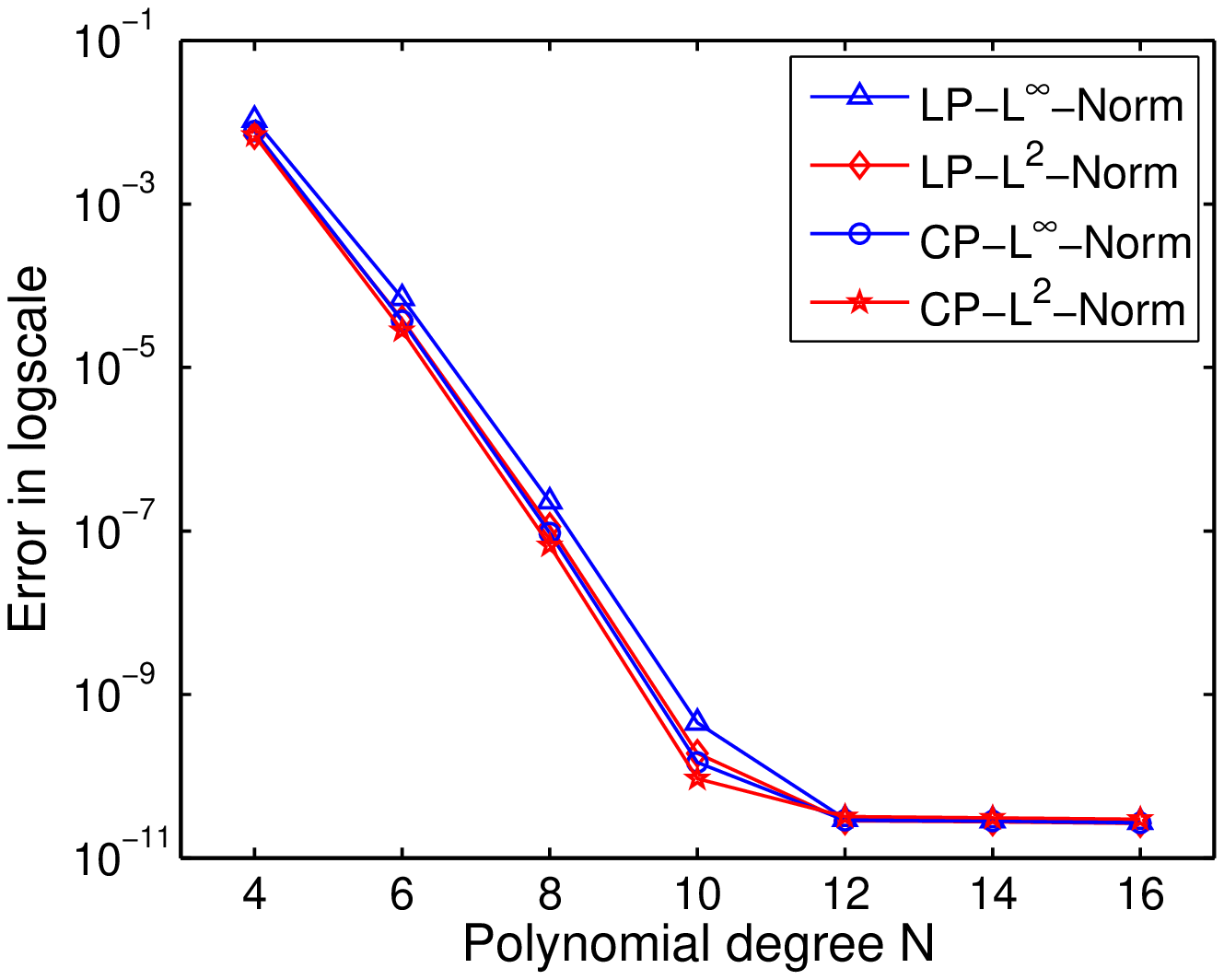}
      \end{minipage}}
    \subfigure[$\alpha=1.7$]{
      \begin{minipage}[t]{0.3\linewidth}
        \centering
        \includegraphics[scale=0.28]{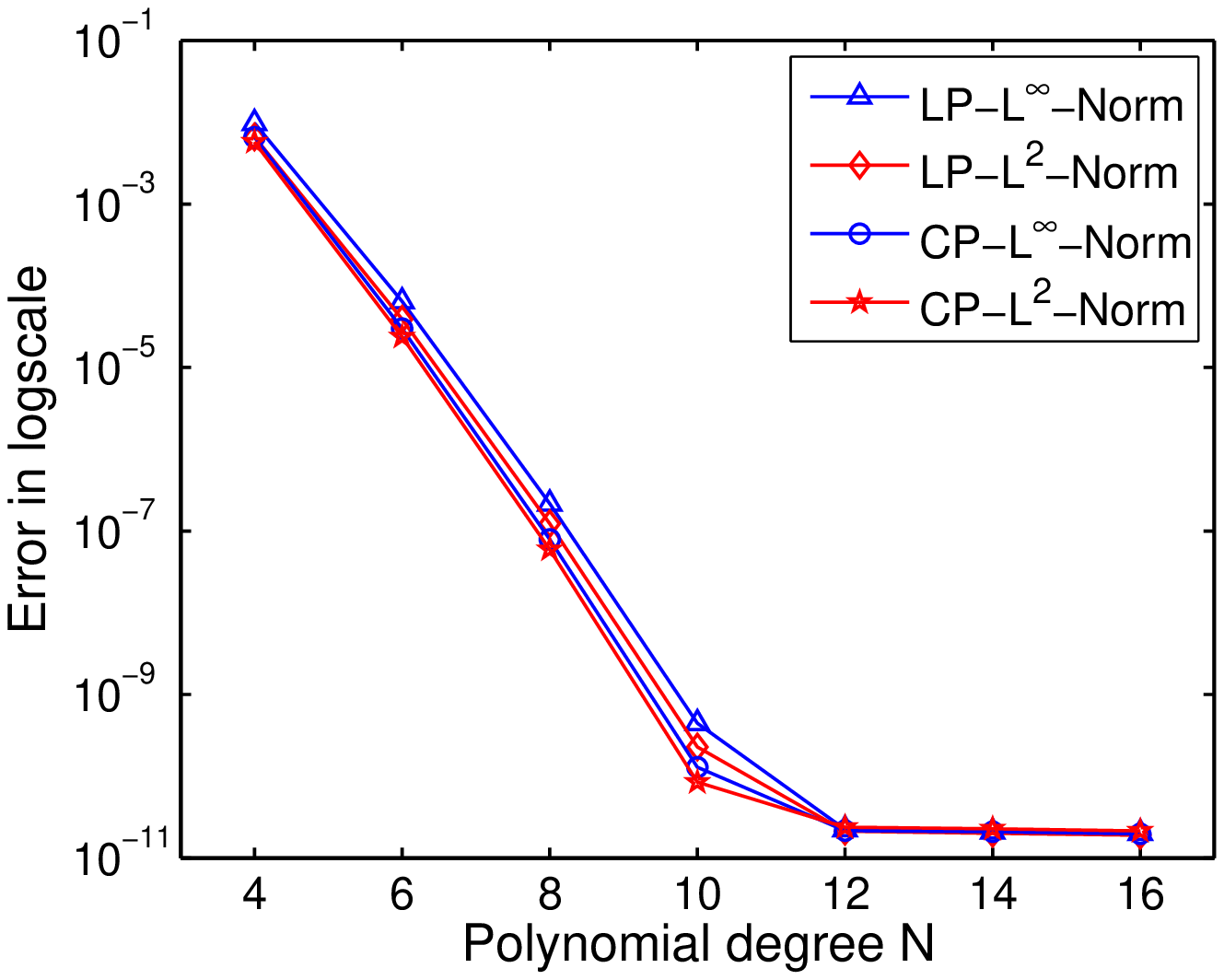}
      \end{minipage}}\hspace{20pt}
    \subfigure[$\alpha=1.9$]{
      \begin{minipage}[t]{0.3\linewidth}
        \centering
        \includegraphics[scale=0.28]{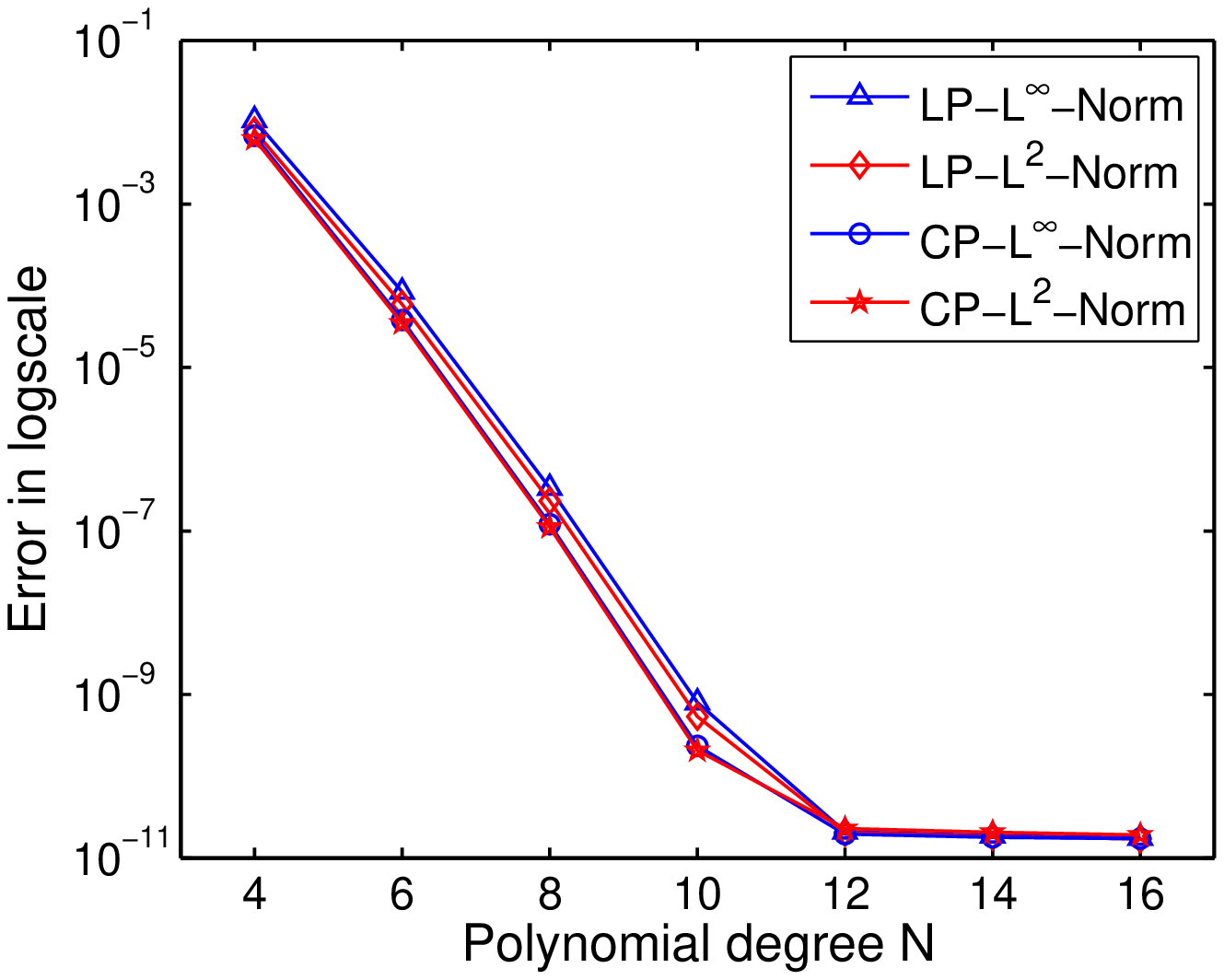}
      \end{minipage}}
    \caption{$L^{\infty}$ and $L^2$ errors to Example \ref{exm:2} approximated on Legendre-Gauss-Lobatto and Chebyshev-Gauss-Lobatto points at $t=1$ for $\alpha=1.1,~1.3,~1.5,~1.7,~1.9$ with $\tau=1/10000$ (LP: Legendre points, CP: Chebyshev points).}\label{fig:c}
\end{figure}
\begin{example}\label{exm:4}
  We approximate the L\'evy-Feller advection-diffusion equation \cite{Liu:07}
  \begin{align*}
      & \frac{\partial u(x,t)}{\partial t}
        =\nabla^{\alpha}u(x,t)-\nabla u(x,t),\quad && -1<x<1,~t>0,\\
      & u(-1,t)=u(1,t)=0,\quad && t>0,\\
      & u(x,0)=\sin(\frac{\pi(x+1)}{2}),\quad && -1\le x\le1,
  \end{align*}
  where $\nabla^{\alpha}u=p\ _{-1}D_x^{\alpha}u+q\ _xD_1^{\alpha}u$, and
  \begin{equation*}
    p=-\frac{\sin((\alpha-\vartheta)\pi/2)}{\sin(\alpha\pi)},\quad
    q=-\frac{\sin((\alpha+\vartheta)\pi/2)}{\sin(\alpha\pi)},\quad |\vartheta|<2-\alpha.
  \end{equation*}
\end{example}
In Figure \ref{fig:f},  (a) shows the numerical solutions
approximated on Legendre-Gauss-Lobatto points for Example
\ref{exm:4} with $\tau=0.1$ at $t=0,~0.2,~0.4,~0.6,~0.8,~1.0$, and
(b) displays the behavior of the solution of the advection-diffusion
process for different $\vartheta$. From Figure \ref{fig:f}(b), we
can see that the parameter $\vartheta$ plays an important role in
the investigation of the non-Fickian transport.
\begin{figure}[h!]
    \subfigure[$\alpha=1.8,~\vartheta=0.1,~N=20$]{
      \begin{minipage}[t]{0.5\linewidth}
        \centering
        \includegraphics[scale=0.4]{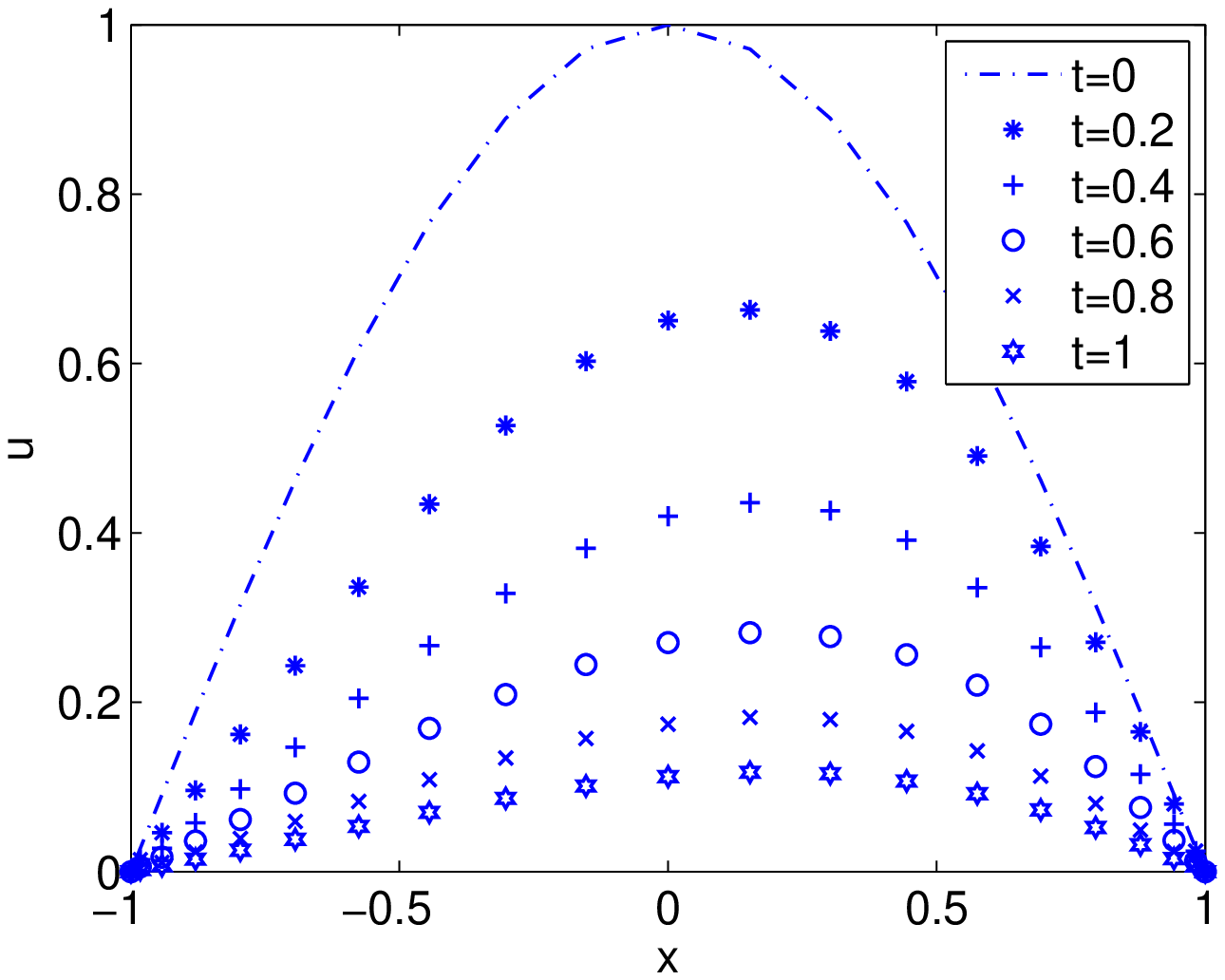}
      \end{minipage}}
    \subfigure[$\alpha=1.6,~t=1,~N=20$]{
      \begin{minipage}[t]{0.5\linewidth}
        \centering
        \includegraphics[scale=0.4]{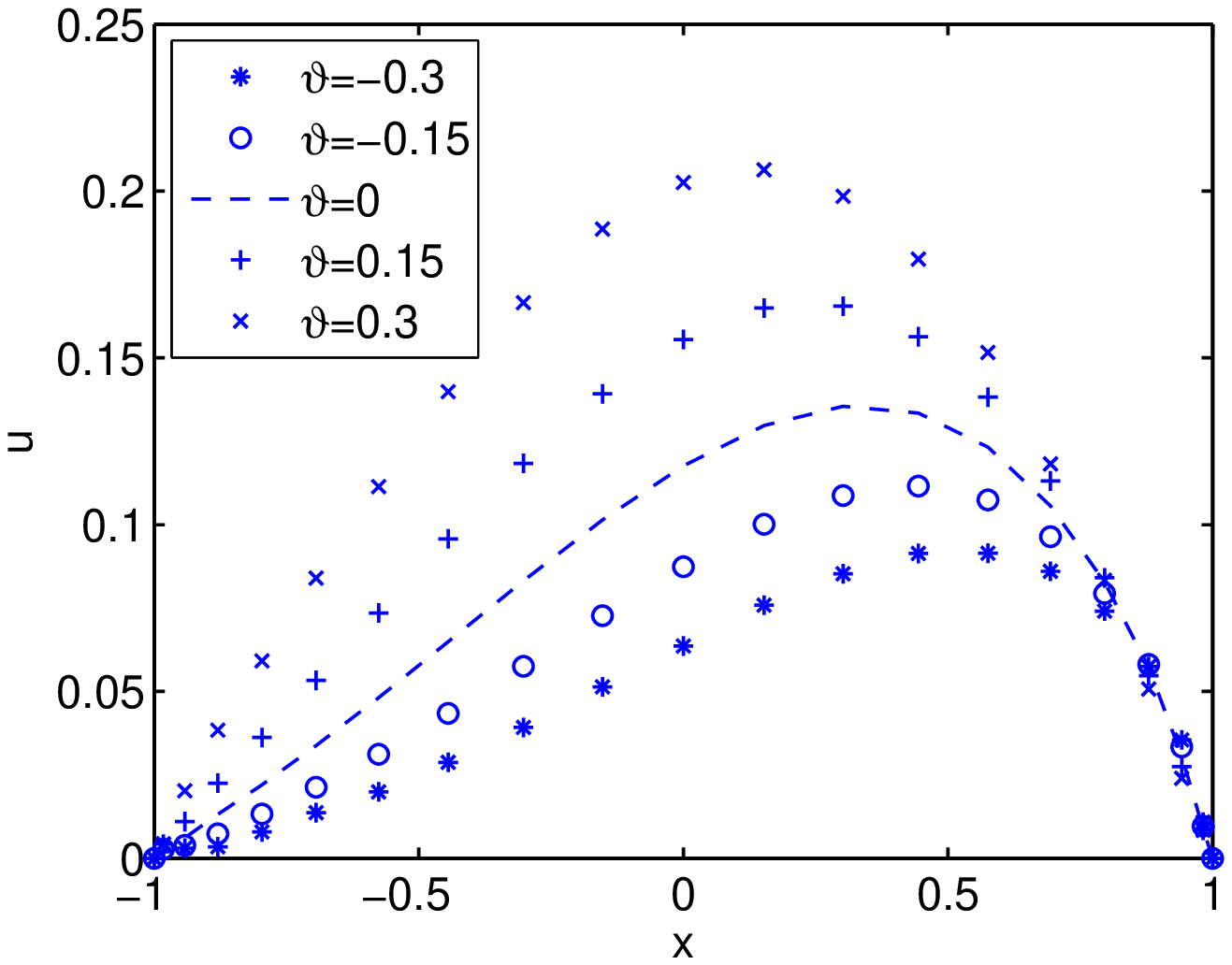}
      \end{minipage}}
    \caption{Approximations on Legendre-Gauss-Lobatto points for Example \ref{exm:4} with $\tau=0.1$.}\label{fig:f}
\end{figure}
\subsection{Examples for Two Dimensions}
For the two dimensional case, the errors are measured in the
$L^{\infty}$ and $L^2$ norms defined by
\begin{equation*}
  \begin{split}
    & L^{\infty}=\max_{0\le r,s\le N}\big|u(x_r,y_s,\cdot)-u_{_N}(x_r,y_s,\cdot)\big|,\\
    & L^2=\Big(\sum_{r,s=0}^N\big|u(x_r,y_s,\cdot)-u_{_N}(x_r,y_s,\cdot)\big|^2\omega_r\omega_s\Big)^{1/2}.
  \end{split}
\end{equation*}
\begin{example}\label{exm:3}
  Setting $u(x,y,t)=(t^2+1)\exp(x^2+y^2)$, $\kappa_{\alpha}=\nu_{\alpha}=1,~p=q=0.5$, and it is a solution of (\ref{eq:2.1}) with the associated Dirichlet boundary conditions.
\end{example}
In Figure \ref{fig:d}, we plot the $L^{\infty}$ and $L^2$ errors to
Example \ref{exm:3} approximated on Legendre-Gauss-Lobatto and
Chebyshev-Gauss-Lobatto points at $t=1$ for
$\alpha=1.1,~1.3,~1.5,~1.7,~1.9$ with $\tau=0.1$. Figure \ref{fig:e}
displays the eigenvalue distribution of the iterative matrix of the
full-discrete scheme of Example \ref{exm:3} for $\alpha=1.5$ with
$\tau=0.1$ and $N=10,~15,~20$, it is noticed that the imaginary
parts of the eigenvalues of the iterative matrix become smaller with
$\alpha$ being large, and the real parts move towards $-1$.
\begin{figure}[h!]\centering
    \subfigure[$\alpha=1.1$]{
      \begin{minipage}[t]{0.3\linewidth}
        \centering
        \includegraphics[scale=0.28]{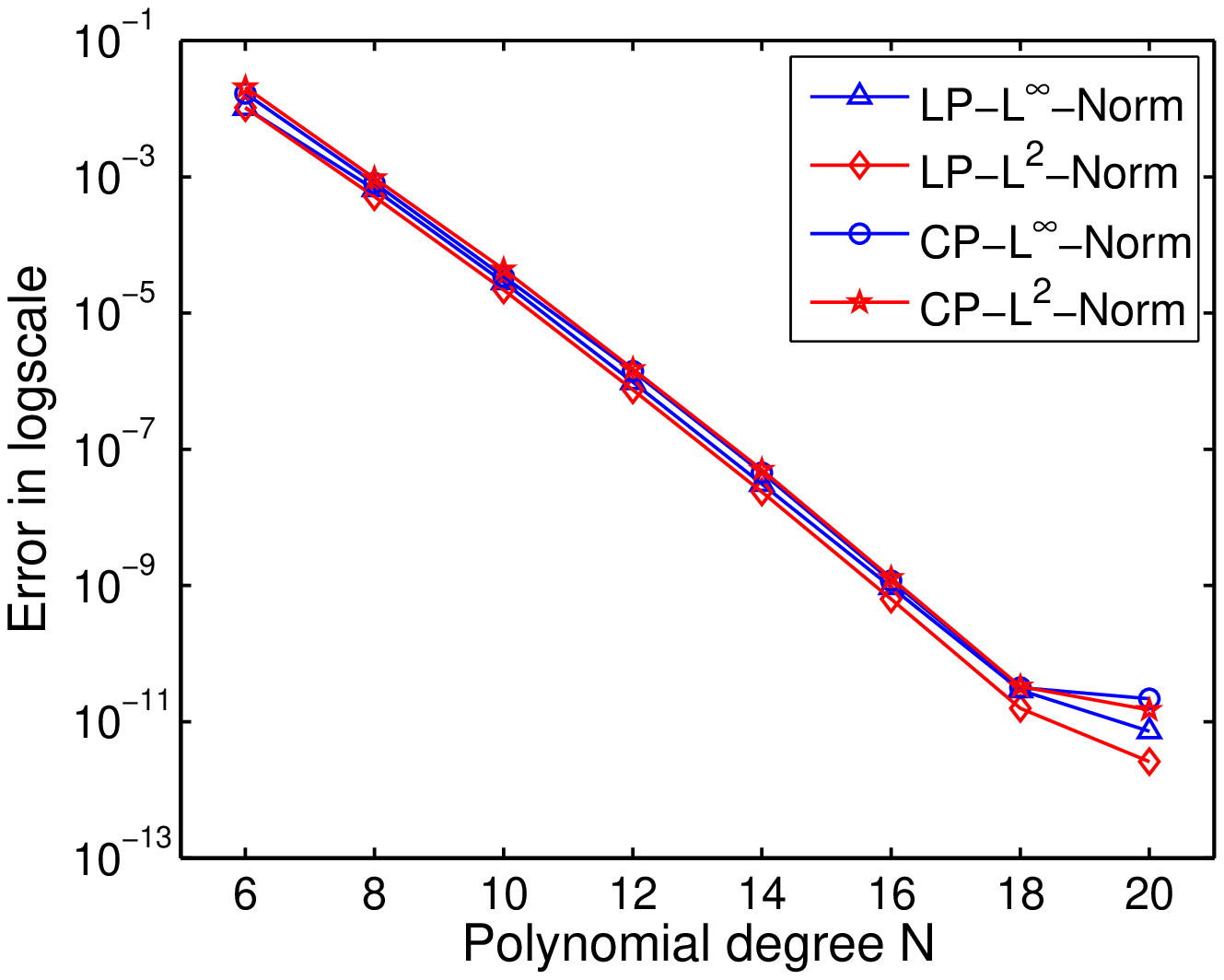}
      \end{minipage}}\hspace{10pt}
    \subfigure[$\alpha=1.3$]{
      \begin{minipage}[t]{0.3\linewidth}
        \centering
        \includegraphics[scale=0.28]{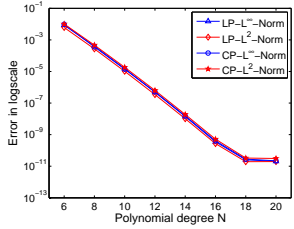}
      \end{minipage}}\hspace{10pt}
    \subfigure[$\alpha=1.5$]{
      \begin{minipage}[t]{0.3\linewidth}
        \centering
        \includegraphics[scale=0.28]{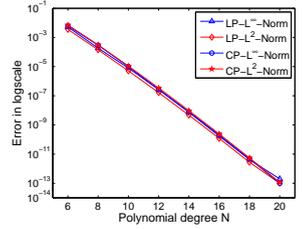}
      \end{minipage}}
    \subfigure[$\alpha=1.7$]{
      \begin{minipage}[t]{0.3\linewidth}
        \centering
        \includegraphics[scale=0.28]{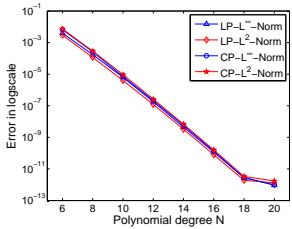}
      \end{minipage}}\hspace{20pt}
    \subfigure[$\alpha=1.9$]{
      \begin{minipage}[t]{0.3\linewidth}
        \centering
        \includegraphics[scale=0.28]{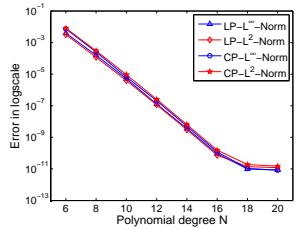}
      \end{minipage}}
    \caption{$L^{\infty}$ and $L^2$ errors to Example \ref{exm:3} approximated on Legendre-Gauss-Lobatto and Chebyshev-Gauss-Lobatto points at $t=1$ for $\alpha=1.1,~1.3,~1.5,~1.7,~1.9$ with $\tau=0.1$ (LP: Legendre points, CP: Chebyshev points).}\label{fig:d}
\end{figure}
\begin{figure}[h!]\centering
    \subfigure[$\alpha=1.2$]{
      \begin{minipage}[t]{0.3\linewidth}
        \centering
        \includegraphics[scale=0.28]{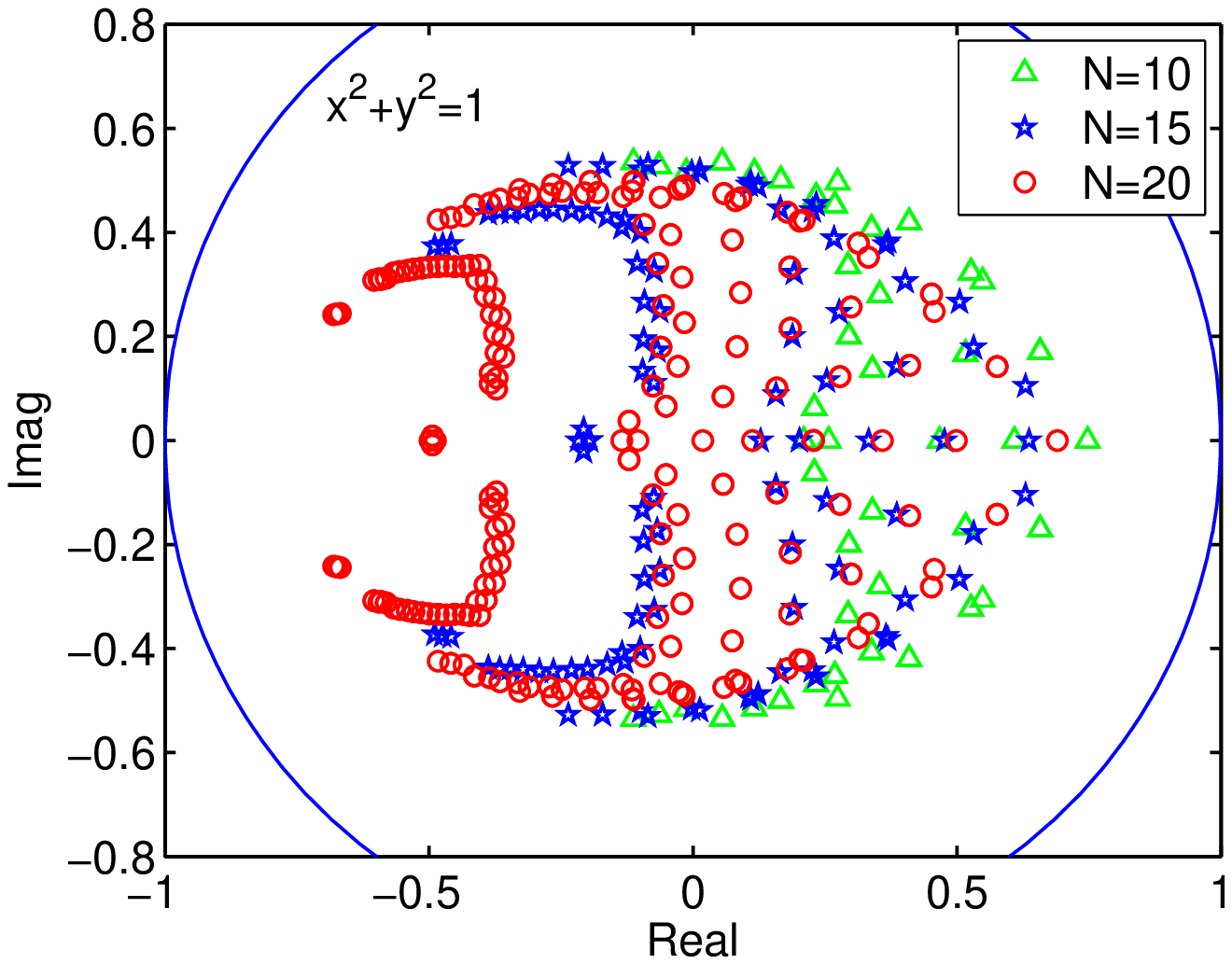}
      \end{minipage}}\hspace{10pt}
    \subfigure[$\alpha=1.5$]{
      \begin{minipage}[t]{0.3\linewidth}
        \centering
        \includegraphics[scale=0.28]{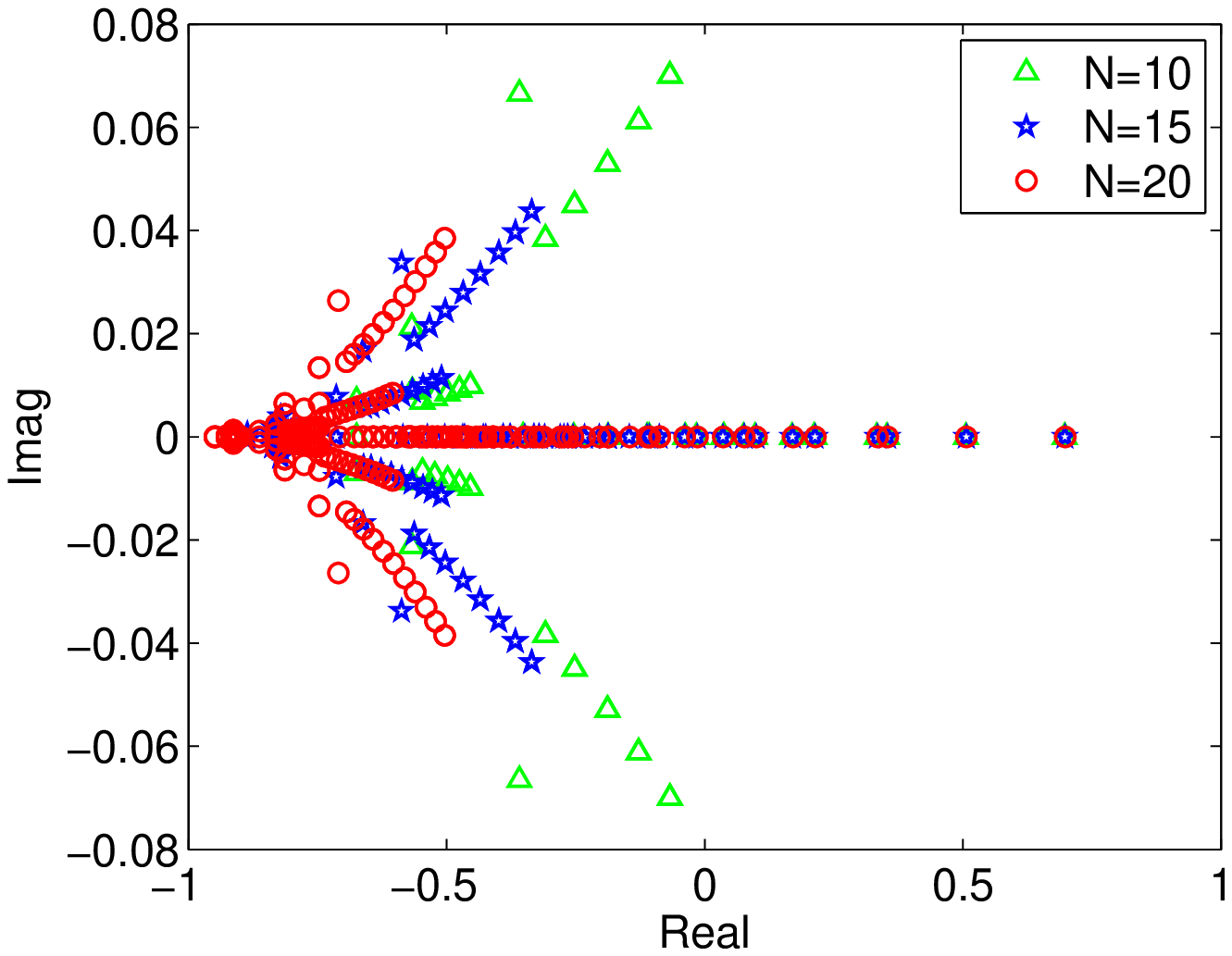}
      \end{minipage}}\hspace{10pt}
    \subfigure[$\alpha=1.8$]{
      \begin{minipage}[t]{0.3\linewidth}
        \centering
        \includegraphics[scale=0.28]{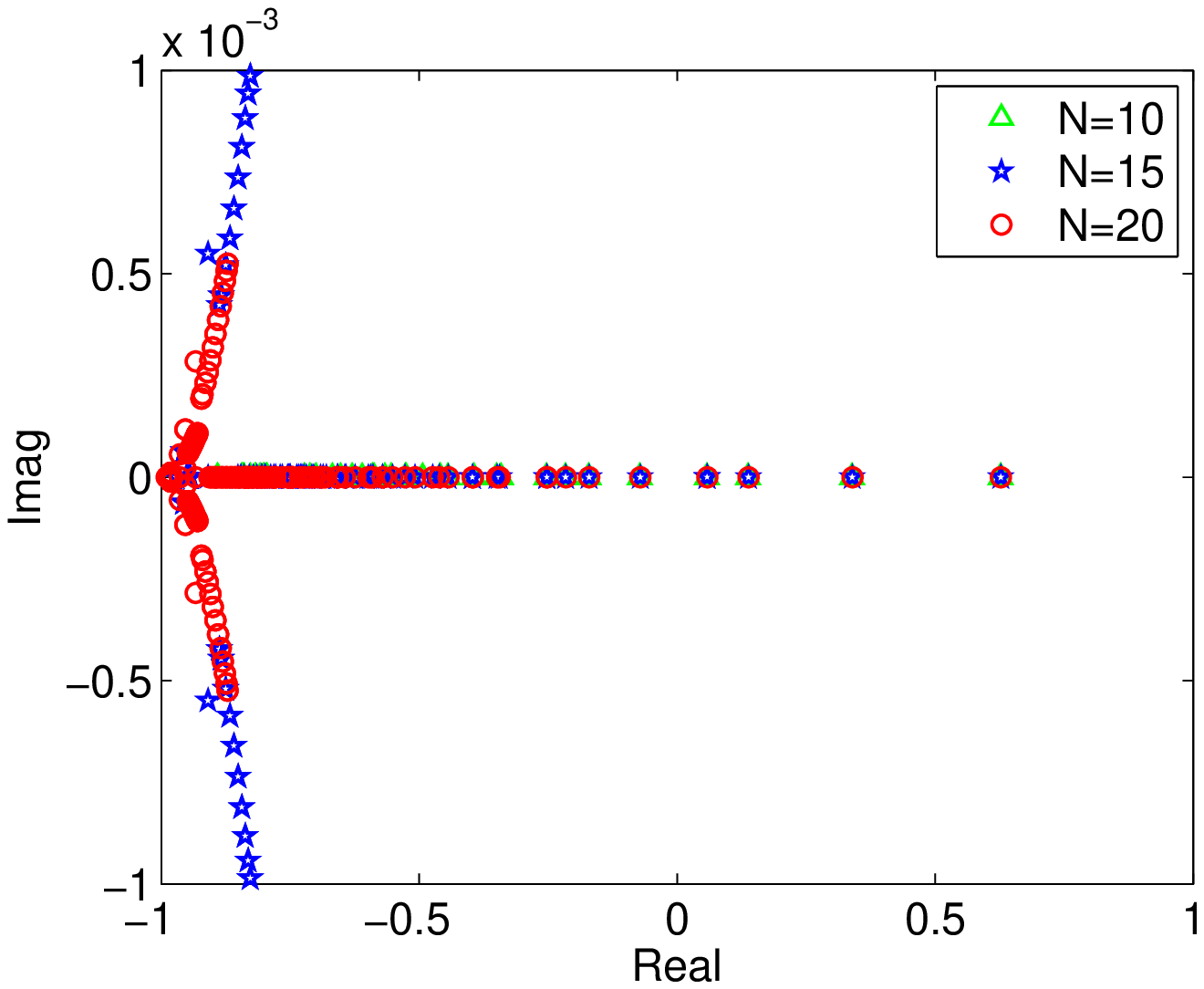}
      \end{minipage}}
    \caption{The eigenvalue distribution of the iterative matrix of the full-discrete scheme of Example \ref{exm:3} approximated on Legendre-Gauss-Lobatto points for $\alpha=1.2,~1.5,~1.8$ with $\tau=0.1$.}\label{fig:e}
\end{figure}
\subsection{Examples for Nonlinear case}
\begin{example}\label{exm:5}
  Setting $u(x,y,t)=(t^2+1)\exp(x^2+y^2)$, $\kappa_{\alpha}=\nu_{\alpha}=1,~p=q=0.5,~m=2$, and it is a solution of (\ref{eq:4.1}) with the corresponding Dirichlet boundary conditions.
\end{example}
\begin{table}[htp]\fontsize{9pt}{12pt}\selectfont
  \begin{center}
  \caption{The $L^{\infty}$ and $L^2$ errors to Example \ref{exm:5} collocated on Legendre-Gauss-Lobatto points at $t=1$ for different $\alpha$ with $\tau=0.1$.}\vspace{5pt}
  \begin{tabular*}{\linewidth}{@{\extracolsep{\fill}}*{1}{l}*{1}{r}*{2}{l}*{1}{r}*{2}{l}}
    \midrule
  $\alpha$ & $N$ & $L^{\infty}$ & $L^2$ & $N$ & $L^{\infty}$ & $L^2$\\
    \toprule
  1.1 &  6 & 1.446E-02 & 1.219E-02 & 14 & 2.707E-07 & 2.415E-07 \\
      &  8 & 1.249E-03 & 1.012E-03 & 16 & 1.309E-08 & 1.195E-08 \\
      & 10 & 8.002E-05 & 7.133E-05 & 18 & 8.331E-10 & 7.543E-10 \\
      & 12 & 5.023E-06 & 4.400E-06 & 20 & 4.800E-09 & 4.962E-09 \\
    \midrule
  $\alpha$ & $N$ & $L^{\infty}$ & $L^2$ & $N$ & $L^{\infty}$ & $L^2$\\
    \toprule
  1.3 &  6 & 2.135E-02 & 2.027E-02 & 14 & 3.403E-07 & 3.590E-07 \\
      &  8 & 1.554E-03 & 1.629E-03 & 16 & 1.759E-08 & 1.829E-08 \\
      & 10 & 1.106E-04 & 1.110E-04 & 18 & 3.095E-09 & 4.351E-09 \\
      & 12 & 6.298E-06 & 6.665E-06 & 20 & 3.232E-09 & 5.049E-09 \\
    \midrule
  $\alpha$ & $N$ & $L^{\infty}$ & $L^2$ & $N$ & $L^{\infty}$ & $L^2$\\
    \toprule
  1.5 &  6 & 2.494E-02 & 2.469E-02 & 16 & 1.962E-08 & 2.043E-08 \\
      &  8 & 1.961E-03 & 1.984E-03 & 18 & 8.622E-10 & 9.002E-10 \\
      & 10 & 1.322E-04 & 1.338E-04 & 20 & 3.496E-11 & 3.649E-11 \\
      & 12 & 7.697E-06 & 7.947E-06 & 22 & 1.284E-12 & 1.360E-12 \\
      & 14 & 4.087E-07 & 4.234E-07 & 24 & 7.194E-14 & 5.820E-14 \\
    \midrule
  $\alpha$ & $N$ & $L^{\infty}$ & $L^2$ & $N$ & $L^{\infty}$ & $L^2$\\
    \toprule
  1.7 &  6 & 2.689E-02 & 2.790E-02 & 14 & 4.536E-07 & 4.645E-07 \\
      &  8 & 2.203E-03 & 2.231E-03 & 16 & 2.200E-08 & 2.225E-08 \\
      & 10 & 1.471E-04 & 1.492E-04 & 18 & 8.563E-10 & 9.453E-10 \\
      & 12 & 8.725E-06 & 8.785E-06 & 20 & 1.263E-10 & 1.260E-10 \\
    \midrule
  $\alpha$ & $N$ & $L^{\infty}$ & $L^2$ & $N$ & $L^{\infty}$ & $L^2$\\
    \toprule
  1.9 &  6 & 2.523E-02 & 3.064E-02 & 14 & 4.357E-07 & 4.812E-07 \\
      &  8 & 2.103E-03 & 2.354E-03 & 16 & 2.255E-08 & 2.343E-08 \\
      & 10 & 1.436E-04 & 1.559E-04 & 18 & 2.996E-09 & 2.889E-09 \\
      & 12 & 8.405E-06 & 9.129E-06 & 20 & 1.208E-09 & 1.297E-09 \\
    \midrule
  \end{tabular*}\label{tab:b}
  \end{center}
\end{table}
The numerical results for Example \ref{exm:5} are summarized in
Table \ref{tab:b}, which show that the spectral collocation method
is also applicable for the nonlinear fractional advection-diffusion
equations.

\section{Conclusion}
The differentiation matrix plays a crucial role in the spectral collocation method, especially for fractional differential equations. We have derived the differentiation matrixes of the left and right Riemann-Liouville and Caputo fractional derivatives and successfully applied the spectral collocation method to handle the numerical solution of the fractional advection-diffusion equation. The stabilities of the semi-discrete and full-discrete scheme in one dimension are theoretically established for the Legendre spectral collocation method. The eigenvalue distributions of the iterative matrix for a variety of systems are computed for further confirming the stability of the numerical schemes in more general cases. The numerical results have shown the efficiency of the proposed methods, moreover, it can be noted that the numerical solution converges exponentially when the exact solution is analytic. The performance of the spectral collocation method for nonlinear fractional advection-diffusion equation is also exhibited. By simulations the interesting physical phenomena for L\'evy-Feller advection-diffusion equation are discovered.
\section*{Acknowledgments}
This work was supported by the Program for New Century Excellent
Talents in University under Grant No. NCET-09-0438, the National
Natural Science Foundation of China under Grant No. 10801067 and No. 11271173, and
the Fundamental Research Funds for the Central Universities under
Grant No. lzujbky-2010-63 and No. lzujbky-2012-k26. WYT thanks for
the help from Li Can, and WHD thanks Chi-Wang Shu for the
inspirations.

\end{document}